\documentclass[runningheads]{llncs}
\usepackage{graphicx}
\usepackage{longtable}
\usepackage{wrapfig}
\usepackage{ucs}
\usepackage{dsfont}
\usepackage{mathrsfs}
\usepackage{mathtools}
\usepackage{amsmath}
\usepackage{amsfonts}
\usepackage{amssymb}
\usepackage{soul}
\usepackage[makeroom]{cancel}
\usepackage{bbm}
\usepackage{mathbbol}
\usepackage[english]{babel}
\usepackage{ucs}
\usepackage[colorlinks=true,linkcolor=blue,citecolor=blue,urlcolor=orange]{hyperref}
\usepackage{longtable}
\usepackage{wasysym}
\usepackage{verbatim}
\usepackage{multirow}
\usepackage{bussproofs}
\usepackage{setspace}
\usepackage{graphicx}
\usepackage{stmaryrd}
\usepackage{forest}
\forestset{smullyan tableaux/.style={for tree={math content},where n children=1{!1.before computing xy={l=\baselineskip},!1.no edge}{},closed/.style={label=below:$\times$},},}
\usepackage{longtable}
\usepackage[all]{xy}
\usepackage{wrapfig}
\usepackage[labelformat=original]{caption}
\usepackage{xcolor}
\usepackage[colorinlistoftodos,prependcaption,textsize=small]{todonotes}
\usepackage[most]{tcolorbox}
\usepackage{enumitem}
\usepackage{xfrac}
\usepackage{tikz}

\newcommand{\invG}{\mathsf{G}_{\mathsf{inv}}}
\newcommand{\KinvG}{\mathbf{K}\mathsf{G}_{\mathsf{inv}}}
\newcommand{\KblG}{\mathbf{K}\mathsf{G}_{\mathsf{bl}}}
\newcommand{\coimplies}{\Yleft}
\newcommand{\Gsquare}{\mathsf{G}^2}
\newcommand{\KGsquare}{\mathbf{K}\mathsf{G}^2}
\newcommand{\KG}{\mathbf{K}\mathsf{G}}
\newcommand{\KbiG}{\mathbf{K}\mathsf{biG}}

\newcommand{\bimodalL}{\mathscr{L}_{\Box,\lozenge}}
\newcommand{\pspace}{\mathsf{PSpace}}
\newcommand{\fullLbilattice}{\mathscr{L}^{\mathsf{bl}}_{\Box,\blacksquare}}
\newcommand{\infcoimplies}{-\!\!\!\sqsubset\!}
\newcommand{\Prop}{\mathtt{Prop}}

\newcommand{\invol}{{\sim_\mathsf{i}}}

\newcommand{\infoGsquare}{\mathsf{G}^{2\pm}_{\blacksquare,\blacklozenge}}
\newcommand{\bimodalLinv}{\mathscr{L}_{\invol,\Box,\lozenge}}
\newcommand{\bimodalLtriangle}{\mathscr{L}_{\triangle,\Box,\lozenge}}

\newcommand{\both}{\mathbf{B}}
\newcommand{\neither}{\mathbf{N}}

\newcommand{\infneg}{{\sim^{\mathbf{I}}}}
\newcommand{\inftriangle}{{\triangle^{\mathbf{I}}}}

\newcommand{\FKinvG}{\KinvG^\mathsf{F}}
\newcommand{\FKblG}{\KblG^\mathsf{F}}

\newcommand{\Fmsf}{\mathsf{F}}
\newcommand{\Gmsf}{\mathsf{G}}

\newcommand{\Rmsf}{\mathsf{R}}
\newcommand{\Smsf}{\mathsf{S}}
\newcommand{\smsf}{\mathsf{s}}
\newcommand{\Tmsf}{\mathsf{T}}

\newcommand{\Wmsf}{\mathsf{W}}

\newcommand{\TKblG}{\mathcal{T}\!(\KblG)}
\newcommand{\real}{\mathsf{rl}}
\newcommand{\Var}{\mathsf{Var}}

\newcommand{\Str}{\mathsf{Str}}

\newcommand{\triangletop}{\triangle^\top}
\newtheorem{convention}{Convention}
\usepackage{comment}
\usepackage{boondox-cal}
\usepackage{rotating}
\begin{document}
\allowdisplaybreaks
\setlength{\jot}{0pt} 
\setlength{\abovedisplayskip}{2pt}
\setlength{\belowdisplayskip}{2pt}
\setlength{\abovedisplayshortskip}{1pt}
\setlength{\belowdisplayshortskip}{1pt}
\setlength{\abovecaptionskip}{5pt plus 3pt minus 2pt}
\setlength{\belowcaptionskip}{5pt plus 3pt minus 2pt}
\title{Simple tableaux for two expansions\\of G\"{o}del modal logic\thanks{The research of Marta B\'ilkov\'a and Thomas Ferguson was supported by the grant 22-01137S of the Czech Science Foundation. The research of Daniil Kozhemiachenko is part of the project INTENDED (ANR-19-CHIA-0014).}}
%
\author{Marta B\'ilkov\'a\inst{1}\orcidID{0000-0002-3490-2083} \and Thomas Ferguson\inst{1,2}\orcidID{0000-0002-6494-1833}\and Daniil Kozhemiachenko\inst{3}\orcidID{0000-0002-1533-8034}}
\authorrunning{B\'ilkov\'a et al.}
\institute{The Czech Academy of Sciences, Institute of Computer Science, Prague\\
\email{bilkova@cs.cas.cz}
\and
Department of Cognitive Science, Rensselaer Polytechnic Institute, Troy, USA\\
\email{tferguson@gradcenter.cuny.edu}
\and
Univ. Bordeaux, CNRS, Bordeaux INP, LaBRI, UMR 5800, Talence, France\\
\email{daniil.kozhemiachenko@u-bordeaux.fr}}
\maketitle              
\begin{abstract}
This paper considers two logics. The first one, $\KinvG$, is an expansion of the G\"{o}del modal logic $\KG$ with the involutive negation~$\invol$ defined as $v(\invol\phi,w)=1-v(\phi,w)$. The second one, $\KblG$, is the expansion of $\KinvG$ with the bi-lattice connectives and modalities. We explore their semantical properties w.r.t.\ the standard semantics on $[0,1]$-valued Kripke frames and define a~unified tableaux calculus that allows for the explicit countermodel construction. For this, we use an alternative semantics with the finite model property. Using the tableaux calculus, we construct a~de\-ci\-sion algorithm and show that satisfiability and validity in $\KinvG$ and $\KblG$ are $\pspace$-complete.
\keywords{constraint tableaux \and G\"{o}del modal logic \and involutive negation \and bi-lattice logics.}
\end{abstract}
\section{Introduction\label{sec:introduction}}
When people evaluate their uncertainty or beliefs, they usually do it in two ways: either quantitatively, by assigning a specific number to the value of the belief, e.g., \emph{‘the rain today is $73\%$ likely’} or qualitatively, by comparing degrees of two beliefs without explicitly mentioning their numerical values as in \emph{‘I think that my wallet is more likely to be in the drawer than in the bag’}. In most cases, a~person does not assign exact values to their beliefs, and thus, reasons qualitatively. Formally, this can be represented via G\"{o}del modal logic $\KG$ or, if one wishes to construe ‘rather’ as ‘strictly more likely’, its expansion with $\triangle$ (Baaz' Delta~\cite{Baaz1996}) or $\coimplies$ (co-implication~\cite{Rauszer1974}\footnote{The symbol $\coimplies$ is due to~\cite{Gore2000}.}) $\KbiG$.

Note that it is natural not only to compare degrees of certainty in two different events but also to say something like \emph{‘my wallet is likely in the drawer’} meaning \emph{‘I think that my wallet is in the drawer rather than elsewhere’}. I.e., the (subjective) likelihood of the wallet being in the drawer is greater than~$\frac{1}{2}$. As G\"{o}del logic ($\Gmsf$) can express only $0$ and $1$, such statements cannot be formalised in $\KG$ or $\KbiG$. However, $\frac{1}{2}$ is expressible with the involutive negation $\invol$ defined as $v(\invol\phi)=1-v(\phi)$: $v(p\leftrightarrow\invol p)=1$ iff $v(p)=\frac{1}{2}$.

In this paper, we will present modal logics that can formalise this kind of reasoning about beliefs. Let us now provide a~broader context of our work.

\noindent
\textbf{G\"{o}del logics with additional negations}
All logics that we consider here expand the propositional G\"{o}del logic with involutive negation ($\invG$) that was introduced in~\cite{EstevaGodoHajekNavara2000}. Adding $\invol$ greatly enhances the expressivity of the G\"{o}del logic. Not only (as we mentioned above) can $\frac{1}{2}$ be defined but also coimplication and the Baaz Delta.
\begin{align*}
\phi\coimplies\chi&\coloneqq\invol(\invol\chi\rightarrow\invol\phi)&\triangle\phi&\coloneqq\mathbf{1}\coimplies(\mathbf{1}\coimplies\phi)
\end{align*}

Moreover, $\invG$ is \emph{paraconsistent} in the following sense: (1) $(p\wedge\invol p)\rightarrow q$ is not valid; (2) $p,\invol p\not\models_{\invG}q$ if $\phi\models_{\invG}\chi$ is interpreted as \emph{‘$v(\phi)\leq v(\chi)$, for each valuation $v$’}. These properties were investigated in~\cite{ErtolaEstevaFlaminioGodoNoguera2015,ConiglioEstevaGispertGodo2021}. Furthermore, $\invG$ is closely connected to $\Gsquare$ that was first discussed in~\cite{Ferguson2014} and independently studied in more detail in~\cite{BilkovaFrittellaKozhemiachenko2021,BilkovaFrittellaKozhemiachenkoMajer2023IJAR}. $\Gsquare$ is an expansion of $\Gmsf$ with a paraconsistent negation $\neg$ whose semantics is defined via two independent valuations $v_1$ (support of truth or positive support) and $v_2$ (support of falsity or negative support). It is shown in~\cite{Ferguson2014} that $\invG$ is obtained from $\Gsquare$ by setting $v_1(p)=1-v_2(p)$. 

\noindent
\textbf{G\"{o}del modal and description logics}
Modal expansions of G\"{o}del logics and their applications have been extensively studied. The $\Box$ and $\lozenge$ fragments\footnote{Note that $\Box$ and $\lozenge$ are not interdefinable in G\"{o}del modal logic.} of $\KG$ were axiomatised in~\cite{CaicedoRodriguez2010}. Hypersequent calculi were constructed in~\cite{MetcalfeOlivetti2009,MetcalfeOlivetti2011} and used to obtain the $\pspace$-completeness of both fragments. $\KG$ in the bi-modal language was axiomatised in~\cite{CaicedoRodriguez2015,RodriguezVidal2021} (for fuzzy and crisp\footnote{A frame $\langle W,R\rangle$ is called \emph{fuzzy} if $R:W\times W\rightarrow[0,1]$ and \emph{crisp} if $R:W\times W\rightarrow\{0,1\}$.} frames, respectively). Moreover, it was shown in~\cite{CaicedoMetcalfeRodriguezRogger2013,CaicedoMetcalfeRodriguezRogger2017} that crisp and fuzzy $\KG$ are $\pspace$-complete. 

The applications of G\"{o}del modal logics and its expansions are well researched. In~\cite{RodriguezTuytEstevaGodo2022}, the completeness of $\mathbf{K45}$ and $\mathbf{KD45}$ extensions of $\KG$ w.r.t.\ non-normalised and normalised possibilistic frames was established. In~\cite{AguileraDieguezFernandez-DuqueMcLean2022}, a~tem\-po\-ral logic expanding the G\"{o}del logic with co-implication was proposed. In addition, paraconsistent expansions of $\KG$ that can be used to reason about contradictory beliefs were proposed and studied in~\cite{BilkovaFrittellaKozhemiachenko2022IJCAR,BilkovaFrittellaKozhemiachenko2023IGPL}.

G\"{o}del description logics (DLs) were proposed in~\cite{BobilloDelgadoGomez-RamiroStraccia2009} to represent graded information in the ontologies and were further studied in~\cite{BobilloDelgadoGomez-RamiroStraccia2012}. Just as the classical description logics are notational variants of the classical logics with the global (universal) modality, G\"{o}del DLs can be considered a notational variant of global G\"{o}del modal logics. Moreover, G\"{o}del DLs are usually equipped with the involutive negation. They differ from \L{}ukasiewicz and Product fuzzy DLs because they are decidable and often have the same complexity as their classical counterparts even for expressive logics~\cite{BorgwardtDistelPenaloza2014DL,BorgwardtDistelPenaloza2014KR,Borgwardt2014PhD,BorgwardtPenaloza2017}.

\noindent
\textbf{Bi-lattice logics}
Our third source of inspiration is the work on bi-lattice logics. Two orders on a~bi-lattice can be interpreted as information order and truth order. The bi-lattice language has two sets of connectives: truth connectives and informational connectives.  Bi-lattices were used to analyse reasoning about beliefs and uncertainty in~\cite{Ginsberg1988} and further studied in~\cite{Rivieccio2010PhD,JansanaRivieccio2012}. In~\cite{JungRivieccio2013} and~\cite{RivieccioJungJansana2017}, the semantics and duality theory of modal bi-lattice logics were studied.

\noindent
\begin{minipage}{.74\linewidth}
\quad~~Note that $\Gsquare$ can be seen as a~bi-lattice logic (with truth connectives). Indeed, even though its semantics was originally defined via two independent valuations $v_1$ and $v_2$ on $[0,1]$, it is possible to combine them into one valuation on the bi-lattice $[0,1]^{\Join}$ (cf.~Fig.~\ref{fig:01join}). Likewise, in its modal extension $\KGsquare$, the semantics of $\Box$ and $\lozenge$ is defined using $\bigwedge$ and $\bigvee$~--- supremum and infimum w.r.t.\ the truth order.
\end{minipage}
\hfill
\begin{minipage}{.24\textwidth}
\resizebox{.92\linewidth}{!}{
\begin{tikzpicture}[>=stealth,relative]
\node (U1) at (0,-1.5) {$(0,1)$};
\node (U2) at (-1.5,0) {$(0,0)$};
\node (U3) at (1.5,0) {$(1,1)$};
\node (U4) at (0,1.5) {$(1,0)$};
\node (U5) at (0.2,0.6) {$\bullet$};
\node (U6) at (0.2,0.4) {$(x,y)$};
\node (tmin) at (-2,-1.5) {};
\node (tmax) at (-2,1.5) {};
\node (imin) at (-1.5,-2) {};
\node (imax) at (1.5,-2) {};
\path[-,draw] (U1) to (U2);
\path[-,draw] (U1) to (U3);
\path[-,draw] (U2) to (U4);
\path[-,draw] (U3) to (U4);
\draw[->,draw] (tmin)  edge node[above,rotate=90] {truth}  (tmax);
\draw[->,draw] (imin)  edge node[below] {information} (imax);
\end{tikzpicture}
}
\captionof{figure}{$[0,1]^{\Join}$}
\label{fig:01join}
\end{minipage}

In either approach, $v_1$ (the first coordinate) can be interpreted as the \emph{support of truth} (the degree to which the statement is \emph{asserted}) and $v_2$ (the second coordinate) as the \emph{support of falsity} (the degree to which the statement is \emph{denied}\footnote{Here, we differentiate between \emph{denials} or \emph{refutations}, i.e., assertions of falsity and \emph{rejections} --- lack of support. Cf.~\cite{BilkovaFrittellaKozhemiachenko2022IJCAR,BilkovaFrittellaKozhemiachenko2023IGPL,BilkovaFrittellaKozhemiachenko2023nonstandard} for more details concerning this distinction.}).

In~\cite{BilkovaFrittellaKozhemiachenko2023nonstandard}, $\infoGsquare$, an expansion of $\Gsquare$ with informational modalities $\blacksquare$ and $\blacklozenge$ was proposed and studied. As expected, the value of $\blacksquare\phi$ was computed using $\bigsqcap$ (informational infimum on $[0,1]^{\Join}$), and the value of $\blacklozenge\phi$ via $\bigsqcup$ (informational supremum). In this paper, we combine $\KGsquare$ and $\infoGsquare$ into one logic with the full set of truth and informational connectives.

\noindent
\textbf{This paper}
In this paper, we present $\KinvG$ and $\KblG$~--- expansions of $\KG$ with involution and with bi-lattice connectives and construct a tableaux calculus that can be used to establish their $\pspace$-completeness. Our contribution is three-fold. First, we establish the result analogous to the one in~\cite{Speranski2022} and construct faithful embeddings between $\KinvG$ and $\KblG$. Second, following~\cite{CaicedoMetcalfeRodriguezRogger2013,CaicedoMetcalfeRodriguezRogger2017}, we present an alternative (and equivalent) semantics for $\KinvG$ and $\KblG$ that possess the finite model property. Third, using this semantics, we construct a~sound and complete tableaux calculus by combining the ideas from~\cite{DAgostino1990}, \cite{Haehnle1994}, and~\cite{Rogger2016phd}.


The rest of the text is as follows. In Section~\ref{sec:KinvG}, we present the G\"{o}del modal logic with involution and discuss its semantical properties. Section~\ref{sec:KblG} is dedicated to the bi-lattice G\"{o}del modal logic $\KblG$. We present its language and semantics and construct faithful embeddings of $\KinvG$ and $\KblG$ into one another. In Section~\ref{sec:Fmodels}, we present an alternative semantics that we use in Section~\ref{sec:TKinvG} to construct a~sound and complete tableaux calculus. Using tableaux, we prove $\pspace$-completeness of $\KinvG$ and $\KblG$. Finally, in Section~\ref{sec:conclusion}, we summarise the results of the paper and provide a~roadmap for future research.
\section{G\"{o}del modal logic with involution\label{sec:KinvG}}
Let us present the language and semantics of $\KinvG$.
\begin{definition}[Frames]\label{def:frames}~
\begin{itemize}[noitemsep,topsep=2pt]
\item A \emph{fuzzy frame} is a tuple $\mathfrak{F}=\langle W,R\rangle$ with $W\neq\varnothing$ and $R:W\times W\rightarrow[0,1]$.
\item A \emph{crisp frame} is a tuple $\mathfrak{F}=\langle W,R\rangle$ with $W\neq\varnothing$ and $R:W\times W\rightarrow\{0,1\}$.
\end{itemize}
\end{definition}
\begin{convention}
Let $\mathfrak{F}=\langle W,R\rangle$ be a frame and $w\in W$. We use the following shorthand: $R(w)\coloneqq\{w':wRw'>0\}$.
\end{convention}
\begin{definition}[$\KinvG$]\label{def:KinvG}
We fix a countable set $\Prop$ and define the language $\bimodalLinv$ using the grammar below.
\begin{align}
\bimodalLinv\ni\phi&\coloneqq p\in\Prop\mid\invol\phi\mid(\phi\wedge\phi)\mid(\phi\rightarrow\phi)\mid\Box\phi\mid\lozenge\phi\label{equ:bimodalLinv}
\end{align}
A $\KinvG$ model is a tuple $\mathfrak{M}=\langle W,R,v\rangle$ with $\langle W,R\rangle$ being a frame and $v:\Prop\times W\rightarrow[0,1]$ (a $\KinvG$ valuation) extended to the complex formulas as follows.
\noindent
\begin{minipage}{.50\linewidth}
\noindent
\begin{align*}
v(\invol\phi,w)&=1\!-\!v(\phi,w)\\
v(\phi\!\rightarrow\!\chi,w)&=\begin{cases}1\text{ if }v(\phi,w)\!\leq\!v(\chi,w)\\v(\chi)\text{ otherwise}\end{cases}
\end{align*}
\end{minipage}
\hfill
\begin{minipage}{.5\linewidth}
\noindent
\begin{align*}
v(\phi\wedge\chi,w)&=\min(v(\phi,w),v(\chi,w))\\
v(\Box\phi,w)&=\inf\limits_{w'\in W}\{wRw'\!\rightarrow\!v(\phi,w')\}\\
v(\lozenge\phi,w)&=\sup\limits_{w'\in W}\{wRw'\!\wedge\!v(\phi,w')\}
\end{align*}
\end{minipage}

We say that $\phi\in\bimodalLinv$ is \emph{$\KinvG$-valid on a pointed frame $\langle\mathfrak{F},w\rangle$} ($\mathfrak{F},w\!\models_{\KinvG}\!\phi$) iff $v(\phi,w)=1$ for any model $\mathfrak{M}$ on $\mathfrak{F}$. $\phi$ is \emph{$\KinvG$ valid on frame $\mathfrak{F}$} ($\mathfrak{F}\!\models_{\KbiG}\!\phi$) iff $\mathfrak{F},w\!\models_{\KinvG}\!\phi$ for any $w\in\mathfrak{F}$.
\end{definition}
\begin{convention}
Given a formula $\phi$, we use $\Prop(\phi)$ to denote the set of variables occurring in it and $|\phi|$ for the number of symbols. We are going to use $\phi\leftrightarrow\chi$ as a shorthand for $(\phi\rightarrow\chi)\wedge(\chi\rightarrow\phi)$ and use the following defined connectives:
\begin{align}
\mathbf{1}&\coloneqq p\rightarrow p&\mathbf{0}&\coloneqq\invol\mathbf{1}&{\sim}\phi&\coloneqq\phi\rightarrow\mathbf{0}\nonumber\\
\phi\vee\chi&\coloneqq\invol(\invol\phi\wedge\invol\chi)&\phi\coimplies\chi&\coloneqq\invol(\invol\chi\rightarrow\invol\phi)&\triangle\phi&\coloneqq\mathbf{1}\coimplies(\mathbf{1}\coimplies\phi)\label{equ:invdefinedconnectives}
\end{align}

Additionally, we will use the following notation:
\begin{itemize}[noitemsep,topsep=2pt]
\item $\KinvG$ is the set of all $\bimodalLinv$-formulas valid over all frames;
\item $\bimodalLtriangle$ is the $\invol$-free fragment of $\bimodalLinv$ with $\triangle$ and $\sim$\footnote{Recall that $\vee$ can be defined in $\Gmsf$ as follows: $\phi\vee\chi\coloneqq((\phi\rightarrow\chi)\rightarrow\chi)\wedge((\chi\rightarrow\phi)\rightarrow\phi)$.} treated as primitive connectives; $\KbiG$ is the set of all $\bimodalLtriangle$-formulas valid over all frames;
\item $\bimodalL$ is the $\triangle$-free fragment of $\bimodalLtriangle$; $\KG$ is the set of all $\bimodalL$-formulas valid over all frames.
\end{itemize}
One can check that if $\phi\in\bimodalL$ is $\KG$-valid, it is also $\KbiG$-valid and $\KinvG$-valid and that if $\phi\in\bimodalLtriangle$ is $\KbiG$-valid, it is $\KinvG$-valid too.

Moreover, given a model $\mathfrak{M}$, we will write $W_\mathfrak{M}$, $R_\mathfrak{M}$, and $v_\mathfrak{M}$ to designate that $W$ is its the set, $R$ is its accessibility relation, and $v$ is its valuation.
\end{convention}
\begin{definition}[Frame definability]
Let $\mathscr{L}$ be a~language. A class $\mathbb{F}$ of frames is \emph{definable by a~for\-mu\-la $\phi\in\mathscr{L}$} iff the following condition holds:
\begin{align*}
\forall\mathfrak{F}:\mathfrak{F}\models_{\KinvG}\phi\text{ iff }\mathfrak{F}\in\mathbb{F}
\end{align*}
\end{definition}

Let us discuss several important semantical properties of $\KinvG$. First of all, it is easy to see that $\Box p\leftrightarrow\invol\lozenge\invol p$ defines the class of crisp frames. On the other hand, $\Box$ and $\lozenge$ are not interdefinable on fuzzy frames.
\begin{theorem}\label{theorem:Boxlozengeinterdefinability}~
\begin{enumerate}[noitemsep,topsep=2pt]
\item $\mathfrak{F}\models_{\KinvG}\Box p\leftrightarrow\invol\lozenge\invol p$ iff $\mathfrak{F}$ is crisp.
\item $\Box$ and $\lozenge$ are not interdefinable on fuzzy frames.
\end{enumerate}
\end{theorem}
\begin{proof}
We begin with 1. Observe that $v(\Box\phi,w)=\inf\{v(\phi,w'):wRw'\}$ and $v(\lozenge\phi,w)=\sup\{v(\phi,w'):wRw'\}$ on crisp frames. From here, it is evident that $\Box p\leftrightarrow\invol\lozenge\invol p$ is valid on crisp frames. For the converse, assume that $\mathfrak{F}$ is not crisp, i.e., there are $w,w'\in\mathfrak{F}$ s.t.\ $wRw'=x$ for some $0<x<1$. Now for every $w''\in W$, set $v(p,w'')=x''$ s.t.\ $x''\geq wRw''$ and $v(p,w'')=1$ only if $wRw''=1$. It is clear that $v(\Box p,w)=1$. On the other hand, we have that $v(\lozenge\invol p,w)>0$ since $0<v(\invol p,w')\leq1-x<1$ and $wRw'=x$. But then, $v(\invol\lozenge\invol p,w)\neq1$, as required.

For 2., take the model below (all variables have the same values as $p$) and note that $v(\Box p,w)=\frac{1}{5}$ and $v(\lozenge p,w)=\frac{1}{4}$.
\begin{align*}
\xymatrix{w'':p=\sfrac{1}{4}~&&\ar[ll]|{\sfrac{2}{3}}~w:p=0~\ar[rr]|{\sfrac{2}{3}}&&~w':p=\sfrac{1}{5}}
\end{align*}
It is straightforward but tedious to show by induction that (i) there is no $\Box$-free formula $\chi$ s.t.\ $v(\chi,w)\in\left\{\frac{1}{5},\frac{4}{5}\right\}$ and (ii) there is no $\lozenge$-free formula $\psi$ s.t.\ $v(\psi,w)\in\left\{\frac{1}{4},\frac{3}{4}\right\}$. (The proof is in Section~\ref{ssec:theorem:Boxlozengeinterdefinability}.)
\end{proof}

Let us remark on the expressivity of $\bimodalLinv$. Recall from the~\nameref{sec:introduction} that we want to formalise statements where the degree of belief is compared to~$\frac{1}{2}$. We use the following statement as an example.
\begin{example}\label{example:formalisation}
\begin{description}[noitemsep,topsep=2pt]
\item[$\mathsf{wal.}$:] \textit{I think that my wallet is likely in the drawer.}
\end{description}
We write $d$ for the ‘my wallet is in the drawer’ and use $\Box$ for ‘I think that’. Now, we need to formalise that the degree of subjective likelihood of $d$ is greater than~$\sfrac{1}{2}$, i.e., we should write a formula that is true at $w$ when $v(\Box d,w)>\frac{1}{2}$.

We formalise $\mathsf{wal.}$ as follows: $\phi_\mathsf{wal.}\coloneqq\invol\triangle(\Box d\rightarrow\invol\Box d)$. A straightforward check reveals that $v(\phi_\mathsf{wal.},w)=1$ iff $v(\Box d,w)>\frac{1}{2}$.
\end{example}

In~\cite{BilkovaFrittellaKozhemiachenko2022IJCAR,BilkovaFrittellaKozhemiachenko2023IGPL,BilkovaFrittellaKozhemiachenko2023nonstandard}, models $\mathfrak{M}=\langle W,R,v\rangle$ are treated as follows: $W$ is the set of sources; the values $wRw'$ represent the degree of trust one source gives to the other; $v$~stands for the statements asserted by the sources. Using this interpretation, \emph{$w'$ is (un)reliable according to $w$} if $wRw'>\frac{1}{2}$ ($wRw'<\frac{1}{2}$, respectively).

We finish the section with a proof that $\bimodalLinv$ defines frames where \emph{all sources are unreliable} (i.e., $wRw'<\frac{1}{2}$ for all $w,w'\in W$) and where \emph{every $w$~has a reliable source $w'$}. We also show that neither of these classes of frames is definable in $\bimodalLtriangle$.
\begin{theorem}\label{theorem:framedefinability}~
\begin{enumerate}[noitemsep,topsep=2pt]
\item $\mathfrak{F}\models_{\KinvG}\invol\triangle(\invol\lozenge\mathbf{1}\rightarrow\lozenge\mathbf{1})$ iff $\sup\{wRw':w,w'\in\mathfrak{F}\}<\frac{1}{2}$.
\item $\mathfrak{F}\models_{\KinvG}\invol\triangle(\lozenge\mathbf{1}\rightarrow\invol\lozenge\mathbf{1})$ iff for every $w\in W$ there is $w'$ s.t.\ $wRw'>\frac{1}{2}$.
\item Let $0<x,x'<1$, and $\mathfrak{F}=\langle\{w\},R\rangle$ and $\mathfrak{F}'=\langle\{w'\},R'\rangle$ be s.t.\ $wRw=x$ and $w'R'w'=x'$. Then $\mathfrak{F}\models_{\KbiG}\phi$ iff $\mathfrak{F}'\models_{\KbiG}\phi$ for all $\phi\in\bimodalLtriangle$.
\end{enumerate}
\end{theorem}
\begin{proof}
We prove 1., give a sketch of the proof of 3., and put the proof of 2.\ in the appendix (Section~\ref{ssec:theorem:framedefinability}). To check 1., assume that $\mathfrak{F}\models_{\KinvG}\invol\triangle(\invol\lozenge\mathbf{1}\rightarrow\lozenge\mathbf{1})$. Then, it means that $v(\invol\lozenge\mathbf{1},w)>v(\lozenge\mathbf{1},w)$ (i.e., $v(\lozenge\mathbf{1},w)<\frac{1}{2}$) for every $w\in\mathfrak{F}$ and every $v$ on $\mathfrak{F}$. But $v(\lozenge\mathbf{1},w)=\sup\{wRw':w'\in W\}$. Hence, $\sup\{wRw':w,w'\in\mathfrak{F}\}<\frac{1}{2}$, as required. For the converse, let $\sup\{wRw':w,w'\in\mathfrak{F}\}\geq\frac{1}{2}$ and assume w.l.o.g.\ that $w\in\mathfrak{F}$ is s.t.\ $\sup\{wRw':w'\in W\}\geq\frac{1}{2}$. Thus, $v(\lozenge\mathbf{1},w)\geq\frac{1}{2}$, i.e., $v(\invol\triangle(\invol\lozenge\mathbf{1}\rightarrow\lozenge\mathbf{1}),w)=0$, as required.

For 3., consider a function $h:[0,1]\rightarrow[0,1]$ s.t.\ $h(0)=0$, $h(x)=x'$, $h(1)=1$ and also satisfying $y\leq z$ iff $h(y)\leq h(z)$ and define $v'(p,w')=h(v(p,w))$. One can check by induction on $\phi\in\bimodalLtriangle$ that if $v(\phi,w)=y$, then $v'(\phi,w)=h(y)$. To obtain the converse direction, use a~function $h':[0,1]\rightarrow[0,1]$ s.t.\ $h(0)=0$, $h(x')=x$, $h(1)=1$ and also satisfying $y\leq z$ iff $h(y)\leq h(z)$ and define $v(p,w)=h'(v'(p,w'))$. The proof is analogous. (Cf.~Section~\ref{ssec:theorem:framedefinability3} for the complete proof.)
\end{proof}

It is immediate from item 3.\ that the classes of frames in items 1.\ and 2. of Theorem~\ref{theorem:framedefinability} are not definable in $\bimodalLtriangle$.
\section{Bi-lattice G\"{o}del modal logic\label{sec:KblG}}
Let us present $\KblG$. We adopt the approach of~\cite{OdintsovWansing2010,OdintsovWansing2017,Drobyshevich2020} and split $\KblG$ valuations into two parts: support of truth ($v_1$) and support of falsity ($v_2$). This will simplify establishing embeddings between $\KinvG$ and $\KblG$ and constructing a unified tableaux calculus for both of them. Additionally, we follow~\cite{BilkovaFrittellaKozhemiachenko2023nonstandard} and use two relations on a frame: $R^+$ to compute the support of truth of modal formulas and $R^-$ to compute their support of falsity. When there is no confusion, we write $v(\phi,w)=(x,y)$ to stand for $v_1(\phi,w)=x$ and $v_2(\phi,w)=y$.
\begin{definition}[$\KblG$]\label{def:KblG}
The language $\fullLbilattice$ is constructed from $\Prop$ via the following grammar.
\begin{align}
\fullLbilattice\ni\phi&\coloneqq p\in\Prop
\left|\begin{matrix}
\neg\phi\mid(\phi\wedge\phi)\mid(\phi\rightarrow\phi)\mid\Box\phi\mid\lozenge\phi\\
-\phi\mid(\phi\sqcap\phi)\mid(\phi\sqsupset\phi)\mid\blacksquare\phi\mid\blacklozenge\phi
\end{matrix}\right.\label{equ:Lbilattice}
\end{align}
A $\KblG$ model is a tuple $\mathfrak{M}=\langle W,R^+,R^-,v_1,v_2\rangle$ s.t.\ $\langle W,R^+,R^-\rangle$ is a bi-relational frame and $v_1,v_2:\Prop\times W\rightarrow[0,1]$ ($\KblG$ valuations) are extended to the complex formulas as follows (below\footnote{We will use this notation throughout the rest of the paper.}, $i,j\in\{1,2\}$, $i\neq j$, $\&\in\{\wedge,\sqcap\}$, $\Rightarrow\in\{\rightarrow,\sqsupset\}$, $\xcancel{\Box}\in\{\Box,\blacksquare\}$, and $\xcancel{\lozenge}\in\{\lozenge,\blacklozenge\}$).
\begin{align*}
v_i(\neg\phi,w)&=v_j(\phi,w)&v_i(-\phi,w)&=1-v_j(\phi,w)\\
v_1(\phi\&\chi,\!w)&=\min(v_1(\phi,w),v_1(\chi,w))&v_1(\phi\!\!\Rightarrow\!\!\chi,\!w)&=\!\begin{cases}1\text{ if }v_1(\phi,\!w)\!\leq\!v_1(\chi,\!w)\\v_1(\chi,w)\text{ else}\end{cases}\\
v_2(\phi\!\wedge\!\chi,\!w)&=\max(v_2(\phi,w),v_2(\chi,w))&v_2(\phi\!\!\rightarrow\!\!\chi,\!w)&=\!\begin{cases}0\text{ if }v_2(\chi,\!w)\!\leq\!v_2(\phi,\!w)\\v_2(\chi,w)\text{ else}\end{cases}\\
v_2(\phi\!\sqcap\!\chi,\!w)&=\min(v_2(\phi,w),v_2(\chi,w))&v_2(\phi\!\!\sqsupset\!\!\chi,\!w)&=\!\begin{cases}1\text{ if }v_2(\phi,\!w)\!\leq\!v_2(\chi,\!w)\\v_2(\chi,w)\text{ else}\end{cases}
\end{align*}
\begin{align*}
v_1(\xcancel{\Box}\phi,w)&=\inf\limits_{w'\in W}\!\{wR^+w'\!\!\rightarrow\!\!v_1(\phi,\!w')\}&v_1(\xcancel{\lozenge}\phi,\!w)&=\sup\limits_{w'\in W}\!\{wR^+w'\wedge v_1(\phi,\!w')\}\\
v_2(\Box\phi,\!w)&=\sup\limits_{w'\in W}\!\{wR^-w'\wedge v_2(\phi,\!w')\}&v_2(\lozenge\phi,\!w)&=\inf\limits_{w'\in W}\!\{wR^-w'\!\!\rightarrow\!\!v_2(\phi,\!w')\}\\
v_2(\blacksquare\phi,w)&=\inf\limits_{w'\in W}\!\{wR^-w'\!\!\rightarrow v_2(\phi,\!w')\}&
v_2(\blacklozenge\phi,w)&=\sup\limits_{w'\in W}\!\{wR^-w'\!\wedge\!v_2(\phi,\!w')\}
\end{align*}

We say that $\phi$ is \emph{$v_1$-valid ($v_2$-valid) on a~pointed frame $\langle\mathfrak{F},\!w\rangle$} ($\mathfrak{F},\!w\!\models^+_{\KblG}\!\phi$ and $\mathfrak{F},w\models^-_{\KblG}\phi$, respectively) iff $v_1(\phi,w)=1$ ($v_2(\phi,w)=0$) for every model $\mathfrak{M}$ on $\mathfrak{F}$. $\phi$ is \emph{strongly valid on $\langle\mathfrak{F},w\rangle$} ($\mathfrak{F},w\models_{\KblG}\phi$) iff it is $v_1$ and $v_2$-valid. $\phi$ is $v_1$-valid ($v_2$-valid, strongly valid) on $\mathfrak{F}$, iff $\mathfrak{F},w\models^+_{\KblG}\phi$ ($\mathfrak{F},w\models^-_{\KblG}\phi$, $\mathfrak{F},w\models_{\KblG}\phi$, respectively) for every $w\in\mathfrak{F}$.
\end{definition}
\begin{convention}[Full bi-lattice language]
Let us expand $\fullLbilattice$. In particular, we define constants $\both$ and $\neither$ s.t.\ $v(\both,w)=(1,1)$ and $v(\neither,w)=(0,0)$ and the involutive negation $\invol$ s.t.\ $v(\invol\phi,w)=(1-v_1(\phi,w),1-v_2(\phi,w))$.
\begin{align}
\mathbf{1}&\coloneqq p\rightarrow p&\mathbf{0}&\coloneqq\neg\mathbf{1}&\phi\!\vee\!\chi&\coloneqq\neg(\neg\phi\!\wedge\!\neg\chi)&\phi\coimplies\chi&\coloneqq\neg(\neg\chi\!\rightarrow\!\neg\phi)\nonumber\\
\both&\coloneqq p\sqsupset p&\neither&\coloneqq-\both&\phi\!\sqcup\!\chi&\coloneqq-(-\phi\!\sqcap\!-\chi)&\phi\infcoimplies\chi&\coloneqq-(-\chi\!\sqsupset\!-\phi)\nonumber\\
\triangle\phi&\coloneqq\mathbf{1}\!\coimplies\!(\mathbf{1}\!\coimplies\!\phi)&{\sim}\phi&\coloneqq\phi\!\rightarrow\!\mathbf{0}&\inftriangle\phi&\coloneqq\both\infcoimplies(\both\infcoimplies\phi)&\infneg\phi&\coloneqq\phi\sqsupset\neither\nonumber\\
\invol\phi&\coloneqq\neg\!-\!\phi\label{equ:bldefinedconnectives}
\end{align}
\end{convention}

As we see, $\fullLbilattice$ expands $\bimodalL$ with a~strong De Morgan negation $\neg$ and connectives defined w.r.t.\ the information order on $[0,1]^{\Join}$ (cf.~Fig.~\ref{fig:01join}): conjunction $\sqcap$, implication $\sqsupset$, a fuzzy version of Fitting's conflation $-$ from~\cite{Fitting1991} and two modalities $\blacksquare$ and $\blacklozenge$. We can construe modalities as aggregation strategies w.r.t.\ contradictory or incomplete information. $\Box$ is \emph{pessimistic} aggregation: the positive support of $\Box\phi$ is calculated via the infimum of the positive support of $\phi$ in the accessible states and the negative support of $\Box\phi$ using the supremum of the negative support of $\phi$. $\blacksquare$ is \emph{cautious} aggregation: positive and negative support of $\blacksquare\phi$ are computed via the infima of positive and negative supports of $\phi$ in the accessible states. Dually, $\lozenge$ and $\blacklozenge$ are \emph{optimistic} and \emph{credulous} aggregations. Note that if support of truth and support of falsity are not independent, $\blacksquare$ and $\Box$ (as well as $\blacklozenge$ and $\lozenge$) coincide. Using two accessibility relations for each modality and interpreting them as degrees of trust, we can model biases one source can have towards the claims of another. If $wR^+w'>wR^-w'$, we can say that $w$ considers $w'$ too sceptical and downplays its denials in comparison to assertions. If $wR^+w'<wR^-w'$, then $w'$ is sensationalist from the point of view of $w$ and the latter is biased in favour of refutations by $w'$. We refer readers to~\cite{BilkovaFrittellaKozhemiachenko2022IJCAR,BilkovaFrittellaKozhemiachenko2023IGPL,BilkovaFrittellaKozhemiachenko2023nonstandard} for a~more detailed discussion of these aggregation strategies and modalities.

Let us now see how we can formalise the statement from Example~\ref{example:formalisation} in $\KblG$.
\begin{example}\label{example:blformalisation}
Recall $\mathsf{wal.}$: \textit{I think that my wallet is likely in the drawer}. We use the notation from Example~\ref{example:formalisation}. Since positive and negative support are independent in $\KblG$, we have two options to interpret ‘likely’: (i)~only the positive support of ‘I think that my wallet is in the drawer’ is greater than $\frac{1}{2}$; (ii) the positive support is greater than $\frac{1}{2}$ and the negative is less than~$\frac{1}{2}$. Here, (i) can be used when the agent has trusted sources that give contradictory information. E.g., the agent asks a~friend about the wallet and receives a response ‘you put it on the middle shelf of the drawer’ (although there are only two shelves). We set: $\phi^{(i)}_\mathsf{wal.}\coloneqq\both\rightarrow\invol\triangle(\Box d\rightarrow\invol\Box d)$.

If no source gives contradictory information, we can use (ii). First, we define $\triangletop\phi\coloneqq\triangle\phi\wedge\neg{\sim}\triangle\phi$. It is easy to see that
\begin{align*}
v(\triangletop\phi,w)&=\begin{cases}(1,0)&\text{ if }v(\phi,w)=(1,0)\\(0,1)&\text{ else}\end{cases}
\end{align*}
We formalise $\mathsf{wal.}$ as follows: $\phi^{(ii)}_\mathsf{wal.}\!\coloneqq\!\triangletop\!(\invol\Box d\!\rightarrow\!\Box d)\!\wedge\!\neg\triangletop\!(\Box d\!\rightarrow\!\invol\Box d)$.

One can check that
\begin{align*}
v(\phi^{(i)}_\mathsf{wal.},w)\!=\!(1,0)&\text{ iff }v_1(\Box d,w)\!>\!0.5
&
v(\phi^{(ii)}_\mathsf{wal.},w)\!=\!(1,0)&\text{ iff }\begin{matrix}v_1(\Box d,w)\!>\!0.5&\text{and}\\v_2(\Box d,w)\!<\!0.5\end{matrix}
\end{align*}
\end{example}

In the remainder of the section, we will show how to embed $\KinvG$ and $\KblG$ into one another. As $\KblG$ is defined on \emph{bi-relational} frames in the general case, we will construct the embeddings between $\KinvG(2)$ ($\KinvG$ on bi-relational frames with two pairs of modalities: $\Box_1$, $\lozenge_1$, $\Box_2$, and $\lozenge_2$\footnote{We will denote this language with $\bimodalLinv(2)$}) and $\KblG$.

Note, however, that $\KblG$ \emph{does not extend $\KinvG$ w.r.t.\ strong validity}. Indeed, consider the frame $\left\langle\left\{w\right\},\left\{wRw=\frac{1}{2}\right\}\right\rangle$. It is easy to check that $v(\Box\mathbf{0}\vee{\sim}\Box\mathbf{0},w)=(1,\frac{1}{2})$ for any $v$. But $\Box\mathbf{0}\vee{\sim}\Box\mathbf{0}$ is $\KinvG$-valid. Moreover, $\neg\lozenge\neg\phi$ is not equivalent to $\Box\phi$ on bi-relational frames (and likewise, $\neg\blacklozenge\neg\phi$ is not equivalent to $\blacksquare\phi$), whence, there are no $\neg$ normal forms. To circumvent this, we introduce two connected translations of $\fullLbilattice$ formulas into $\bimodalLinv$.
\begin{definition}\label{def:+-translation}
Let $\phi\in\fullLbilattice$. For each $p\in\Prop(\phi)$, we set $p^*$ to be a~fresh variable and define the following translations $^\oplus$ and $^\ominus$.
\begin{align*}
p^\oplus&=p&(\neg\phi)^\oplus&=\phi^\ominus&(-\!\phi)^\oplus&=\invol(\phi^\ominus)\\
p^\ominus&=p^*&(\neg\phi)^\ominus&=\phi^\oplus&(-\!\phi)^\ominus&=\invol(\phi^\oplus)\\
(\phi\&\chi)^\oplus&=\phi^\oplus\!\wedge\!\chi^\oplus&(\phi\wedge\chi)^\ominus&=\phi^\ominus\!\vee\!\chi^\ominus&(\phi\sqcap\chi)^\ominus&=\phi^\ominus\!\wedge\!\chi^\ominus\\
(\phi\Rightarrow\chi)^\oplus&=\phi^\oplus\!\rightarrow\!\chi^\oplus&(\phi\rightarrow\chi)^\ominus&=\chi^\ominus\!\coimplies\!\phi^\ominus&(\phi\sqsupset\chi)^\ominus&=\phi^\ominus\!\rightarrow\!\chi^\ominus\\
(\xcancel{\Box}\phi)^\oplus&=\Box_1\phi^\oplus&(\Box\phi)^\ominus&=\lozenge_2\phi^\ominus&(\blacksquare\phi)^\ominus&=\Box_2\phi^\ominus\\
(\xcancel{\lozenge}\phi)^\oplus&=\lozenge_1\phi^\oplus&(\lozenge\phi)^\ominus&=\Box_2\phi^\ominus&(\blacklozenge\phi)^\ominus&=\lozenge_2\phi^\ominus
\end{align*}

Let $\phi\in\bimodalLinv(2)$ and $\heartsuit,\overline{\heartsuit}\in\{\Box,\lozenge\}$ with $\heartsuit\neq\overline{\heartsuit}$. Then $\phi^{\Join}$ is the result of replacement of $\heartsuit_1$'s with $\heartsuit$'s and $\heartsuit_2$'s with $\neg\overline{\heartsuit}\neg$'s.
\end{definition}
\begin{example}[How $^\oplus$, $^\ominus$, and $^{\Join}$ work]\label{example:+-translation}
Consider $\phi=\Box(p\rightarrow\blacklozenge\neg q)\vee-\!r$. Applying Definition~\ref{def:+-translation}, we obtain
\begin{align*}
\phi^\oplus&=\Box_1(p\rightarrow\lozenge_1 q^*)\vee\invol r^*&\phi^\ominus&=\lozenge_2(\lozenge_2q\coimplies p^*)\wedge\invol r
\end{align*}

Now take $\chi=\Box_1(p\wedge q)\rightarrow\lozenge_2p$. We obtain $\chi^{\Join}=\Box(p\wedge q)\rightarrow\neg\Box\neg p$.

One can see that $^\oplus$ and $^\ominus$ encode, respectively, support of truth and support of falsity of $\fullLbilattice$ formulas in $\bimodalLinv$ (cf.~Definitions~\ref{def:KinvG} and~\ref{def:KblG}). 
\end{example}
\begin{theorem}\label{theorem:embedding}
Let $\phi\in\bimodalLinv(2)$, $\chi\in\fullLbilattice$, and let $\Box_1$ and $\Box_2$ be associated with $R^+$ and $R^-$, respectively. Then, for every bi-relational pointed frame $\langle\mathfrak{F},w\rangle=\langle\langle W,R^+R^-\rangle,w\rangle$, it holds that
\begin{enumerate}[noitemsep,topsep=2pt]
\item $\mathfrak{F},w\models_{\KinvG}\phi$ iff $\mathfrak{F},w\models_{\KblG}\both\rightarrow\phi^{\Join}$;
\item $\mathfrak{F},w\models_{\KblG}\chi$ iff $\mathfrak{F},w\models_{\KinvG(2)}\chi^\oplus\!\wedge\!{\sim}(\chi^\ominus)$.
\end{enumerate}
\end{theorem}
\begin{proof}
For 1., note that it suffices to check that $\mathfrak{F},w\models_{\KinvG}\phi$ iff $\mathfrak{F},w\models^+_{\KblG}\phi^{\Join}$ since $v(\both,w)=(1,1)$. The proof is a straightforward induction (cf.~Section~\ref{ssec:theorem:embedding1}).

To obtain 2., we show that the following statements hold (cf.~Section~\ref{ssec:theorem:embedding}).
\begin{align*}
(a)~\mathfrak{F},w\models_{\KblG}\chi^\oplus&\text{ iff }\mathfrak{F},w\models^+_{\KblG}\chi&(b)~\mathfrak{F},w\models_{\KblG}{\sim}(\chi^\ominus)\text{ iff }\mathfrak{F},w\models^-_{\KblG}\chi
\end{align*}
The result follows since $\chi$ is strongly valid on $\langle\mathfrak{F},w\rangle$ iff it is both $v_1$- and $v_2$-valid.
\end{proof}
\section{Alternative semantics\label{sec:Fmodels}}
Now, let us present alternative semantics for $\KinvG$ and $\KblG$. Note that in~\cite{CaicedoMetcalfeRodriguezRogger2017}, a unified procedure is given to obtain semantics with the finite model property for every logic whose connectives are order-based, i.e., have first-order definitions via the lattice join and meet. However, $\invol$ is not order-based, so we cannot apply the results of~\cite{CaicedoMetcalfeRodriguezRogger2017}. Instead, we adapt the less general construction from~\cite{CaicedoMetcalfeRodriguezRogger2013}.
\begin{definition}[$\Fmsf$-models for $\KinvG$]\label{def:F-KinvG}~

\noindent
A \emph{$\KinvG$ $\Fmsf$-model} is a tuple $\mathfrak{M}=\langle W,R,T,v\rangle$ with $\langle W,R\rangle$ being a~frame and $T:W\rightarrow\mathcal{P}_{<\omega}([0,1])$ be s.t.\ $\{0,1\}\subseteq T(w)$ for all $w\in W$. The valuation is extended to the complex formulas as in $\KinvG$ (Definition~\ref{def:KinvG}) in the cases of propositional connectives, and in the modal cases, as follows.
\begin{align*}
v(\Box\phi,w)&=\max\{x\in T(w)\mid x\leq\inf\limits_{w'\in W}\{wRw'\rightarrow v(\phi,w')\}\}\\
v(\lozenge\phi,w)&=\min\{x\in T(w)\mid x\geq\sup\limits_{w'\in W}\{wRw'\wedge v(\phi,w')\}\}
\end{align*}
\end{definition}
\begin{definition}[$\Fmsf$-models for $\KblG$]\label{def:F-KblG}~

\noindent
A \emph{$\KblG$ $\Fmsf$-model} is a tuple $\mathfrak{M}=\langle W,R^+,R^-,T_1,T_2,v_1,v_2\rangle$ with $\langle W,R^+,R^-\rangle$ being a~bi-relational frame and $T_1,T_2:W\rightarrow\mathcal{P}_{<\omega}([0,1])$ be s.t.\ $\{0,1\}\subseteq T^\pm(w)$ for all $w\in W$. The valuations are extended to the complex formulas as in $\KblG$ (Definition~\ref{def:KblG}) for propositional connectives, and as follows for modalities.
\begin{align*}
v_1(\xcancel{\Box}\phi,w)&=\max\{x\in T_1(w)\mid x\leq\inf\limits_{w'\in W}\{wR^+w'\rightarrow v_1(\phi,w')\}\}\\
v_1(\xcancel{\lozenge}\phi,w)&=\min\{x\in T_1(w)\mid x\geq\sup\limits_{w'\in W}\{wR^+w'\wedge v_1(\phi,w')\}\}\\
v_2(\Box\phi,w)&=\min\{x\in T_2(w)\mid x\geq\sup\limits_{w'\in W}\{wR^-w'\wedge v_2(\phi,w')\}\}\\
v_2(\lozenge\phi,w)&=\max\{x\in T_2(w)\mid x\leq\inf\limits_{w'\in W}\{wR^-w'\rightarrow v_2(\phi,w')\}\}\\
v_2(\blacksquare\phi,w)&=\max\{x\in T_2(w)\mid x\leq\inf\limits_{w'\in W}\{wR^-w'\rightarrow v_2(\phi,w')\}\}\\
v_2(\blacklozenge\phi,w)&=\min\{x\in T_2(w)\mid x\geq\sup\limits_{w'\in W}\{wR^-w'\wedge v_2(\phi,w')\}\}
\end{align*}
\end{definition}

We will use $\FKinvG$ and $\FKblG$ to denote the sets of, respectively, $\bimodalLinv$ and $\fullLbilattice$ formulas valid in all $\Fmsf$-models. Let us recall the definition of generated submodels (for brevity, we define generated $\KinvG$ submodels; the definition of generated $\KblG$ submodels can be produced similarly).
\begin{definition}\label{def:generatedmodel}~
\begin{itemize}[noitemsep,topsep=2pt]
\item A frame $\widehat{\mathfrak{F}}=\langle\widehat{W},\widehat{R}\rangle$ is \emph{a~subframe} of $\mathfrak{F}=\langle W,R\rangle$ if $\widehat{W}\subseteq W$ and $\widehat{R}$ is a~restriction of $R$ on $\widehat{W}$.
\item $\widehat{\mathfrak{M}}=\langle\widehat{F},\widehat{T},\widehat{v}\rangle$ is \emph{a~submodel of} $\mathfrak{M}=\langle\mathfrak{F},T,v\rangle$ if $\widehat{\mathfrak{F}}$ is a~subframe of $\mathfrak{F}$ and $\widehat{v}$ and $\widehat{T}$ are restrictions of $v$ and $T$ on $\widehat{W}$.
\item $\widehat{\mathfrak{M}}$ is \emph{generated by $X\subseteq W$} iff it is the smallest submodel containing $X$ s.t.\ if $w\in\widehat{W}$ and $w'\in R(w)$, then $w'\in\widehat{W}$, as well.
\end{itemize}
\end{definition}
\begin{convention}
A model $\mathfrak{M}$ is \emph{tree-like} if its underlying frame is a~tree. We use $\mathcal{h}(\mathfrak{M})$ to denote its \emph{height}.
\end{convention}

In the remainder of the section, we show that these semantics are equivalent to those in Definitions~\ref{def:KinvG} and~\ref{def:KblG}. The proof is the same as that of~\cite[Theorem~1]{CaicedoMetcalfeRodriguezRogger2013} but we need to alter the formulation of the following statement.
\begin{proposition}\label{prop:lemma1}
Let $\mathbf{L}\!\in\!\{\KinvG,\FKinvG\}$, $\mathfrak{M}$ be an $\mathbf{L}$-model, and $\phi\in\bimodalLinv$.
\begin{enumerate}[noitemsep,topsep=2pt]
\item[$(a)$] If $\langle\widehat{\mathfrak{F}},\!\widehat{v}\rangle$ is a generated submodel of $\langle\mathfrak{F},\!v\rangle$, then $\widehat{v}(\phi,w)\!=\!v(\phi,w)$ for all $w\!\in\!\widehat{\mathfrak{F}}$.
\item[$(b)$] For all $w\!\in\!\mathfrak{M}$, there is a tree-like $\mathbf{L}$-submodel $\widehat{\mathfrak{M}}$ of $\mathfrak{M}$ generated by $w$ s.t.\ $v(\phi,\!w)\!=\!\widehat{v}(\phi,\!w)$ ($v_1(\phi,\!w)\!=\!\widehat{v_1}(\phi,\!w)$ and $v_2(\phi,\!w)\!=\!\widehat{v_2}(\phi,\!w)$) and $\mathcal{h}(\widehat{\mathfrak{M}})\!\leq\!|\phi|$.
\item[$(c)$]  Let $g:[0,1]\rightarrow[0,1]$ be s.t.\ $g(0)=0$, $g(1)=1$, \underline{$1-g(x)=g(1-x)$}, and $x\leq x'$ iff $g(x)\leq g(x')$. Let further $\mathfrak{N}$ be s.t.\ $W_\mathfrak{M}=W_\mathfrak{N}$, $wR_\mathfrak{N}w'=g(wR_\mathfrak{M}w')$, $T_\mathfrak{N}(w)=g(T_\mathfrak{M}(w))$, and $v_\mathfrak{N}(p,w)=g(v_\mathfrak{M}(p,w))$ for all $w,w'\in\mathfrak{M}$ and $p\in\Prop$. Then $v_\mathfrak{N}(\phi,w)=g(v_\mathfrak{M}(\phi,w))$ for each $\phi\in\bimodalLinv$.
\end{enumerate}
\end{proposition}
The proof is standard, whence, we omit it here. Note, however, that because of $\invol$, we need the underlined part of $(c)$ that was not present in the original statement. Otherwise, we could have $g(\frac{1}{2})=\frac{2}{3}$ which would fail $(c)$ since $v_\mathfrak{M}(p\leftrightarrow\invol p,w)=1$ if $v_\mathfrak{M}(p,w)=\frac{1}{2}$ but $v_\mathfrak{N}(p\leftrightarrow\invol p,w)=\frac{1}{3}$ since $v_\mathfrak{N}(p,w)=\frac{2}{3}$.

The next lemma combines Lemmas~2 and~3 from~\cite{CaicedoMetcalfeRodriguezRogger2013}. Again, the proof is the same as in~\cite{CaicedoMetcalfeRodriguezRogger2013} and we omit it.
\begin{proposition}\label{prop:lemma2-3}~
\begin{enumerate}[noitemsep,topsep=2pt]
\item For any tree-like $\FKinvG$-model $\mathfrak{M}$ of finite height with root $w$ there is a tree-like $\KinvG$-model $\mathfrak{M}'$ with root $w'$ s.t.\ $v_\mathfrak{M}(\phi,w)=v_{\mathfrak{M}'}(\phi,w')$ for each $\phi\in\bimodalLinv$. In addition, if $\mathfrak{M}$ is crisp, $\mathfrak{M}'$ is crisp too.
\item Let $\Sigma\subseteq\bimodalLinv$ contain $\mathbf{1}$ and $\mathbf{0}$, be finite and closed under taking subformulas and let $\mathfrak{M}=\langle\mathfrak{F},v\rangle$ be a~tree-like $\KinvG$-model of finite height with root $w$. Then there is a~tree-like $\FKinvG$-model $\widehat{\mathfrak{M}}=\langle\widehat{F},\widehat{T},\widehat{v}\rangle$ with root $w$ s.t.\ $\widehat{\mathfrak{F}}\subseteq\mathfrak{F}$, $|W|\leq|\Sigma|^{\mathcal{h}(\mathfrak{M})}$, $|\widehat{T}(w')|\leq|\Sigma|$, and $v(\phi,w)=\widehat{v}(\phi,w)$ for every $\phi\in\Sigma$ and $w'\in\widehat{\mathfrak{M}}$. In addition, if $\mathfrak{M}$ is crisp, $\widehat{\mathfrak{M}}$ is crisp too.
\end{enumerate}
\end{proposition}

Using Propositions~\ref{prop:lemma1} and~\ref{prop:lemma2-3}, it is immediate to obtain that $\KinvG$ and $\FKinvG$ coincide. Similarly, one can reformulate Propositions~\ref{prop:lemma1} and~\ref{prop:lemma2-3} for $\KblG$ and $\FKblG$ (but now $\{\mathbf{0},\mathbf{1},\both,\neither\}\in\Sigma$) and get that $\FKblG$ and $\KblG$ also coincide.
\section{Tableaux\label{sec:TKinvG}}
In Section~\ref{sec:Fmodels}, we used the result of~\cite{CaicedoMetcalfeRodriguezRogger2013} to obtain that $\KinvG$ and $\KblG$ have the finite model property w.r.t.\ $\Fmsf$-models. This, however, does not give a~decision algorithm. Thus, we will construct a tableaux calculus that allows us to extract countermodels from open branches and use it to define a~decision procedure that takes polynomial space. In~\cite{Rogger2016phd}, tableaux for $\KG$ are presented but they use the fact that all $\bimodalL$ connectives are order-based which is not the case for $\bimodalLinv$. To incorporate $\invol$, we follow~\cite{Haehnle1994} and provide a~\emph{constraint tableaux calculus}. Moreover, we follow the idea of~\cite{DAgostino1990} and use independent constraints for $v_1$ and~$v_2$. Thus, we will be able to use one calculus for both $\KblG$ and $\KinvG$.
\begin{definition}[Constraints]\label{def:structure}
We fix a~countable sets $\Wmsf=\{w,w',w_0,\ldots\}$ of \emph{state-labels} and $\Var=\{c,d,e,c',\ldots\}$ of variables ranging over $[0,1]$ and define $\Tmsf=\{t^i_j(w)\!:\!w\!\in\!\Wmsf,i\!\in\!\mathbb{N},j\in\{0,1\}\}\cup\{0,1\}$. A~\emph{value term} $t$ is either a member of $\Var$ or $\Tmsf$, or a~construction $1-t'$ with $t'$ being a~term.

Now let $\mathbf{i}\in\{\mathbf{1},\mathbf{2}\}$, $\Smsf\in\{\Rmsf^+,\Rmsf^-\}$, $w,w'\in\Wmsf$, $t$ and $t'$ be terms, $t(w)\in\Tmsf$, and $\triangledown\in\{\leqslant,<,\geqslant,>\}$. A~\emph{constraint} has one of the following forms:
\begin{align*}
w:\mathbf{i}:\phi\triangledown t&&w:\mathbf{i}:\phi=t(w)&&w\Smsf w'\triangledown w:\mathbf{i}:\phi&&t\triangledown t'&&w\Smsf w'\triangledown t
\end{align*}
Constraint $\mathfrak{X}\triangledown\mathfrak{X}'$ consists of two \emph{structures} (we denote the set of structures with~$\Str$).
\end{definition}

\begin{definition}[$\TKblG$ --- tableaux for $\KblG$]\label{def:TKblG}
A \emph{tableau} is a downward-branching tree whose branches are sets of constraints. Each branch can be extended by a rule below. For convenience, we also include the rules for~$\invol$.

Bars denote branching; $\mathbf{i},\mathbf{j}\!\in\!\{\mathbf{1},\mathbf{2}\}$, $\mathbf{i}\!\neq\!\mathbf{j}$; $\blacktriangledown,\triangledown\!\in\!\{\leqslant,<,\geqslant,>\}$, if $\triangledown$ is $\leqslant$, then $\blacktriangledown$ is $\geqslant$ (likewise for $<$); $\triangleright\!\in\!\{\geqslant,>\}$; $\mathfrak{Y}\!\approx\!\mathfrak{X}$ stands for ‘$\mathfrak{Y}\!\leqslant\!\mathfrak{X}$ and  $\mathfrak{Y}\!\geqslant\!\mathfrak{X}$’; $c$, $w'$, $t_0(w)$, and $t_1(w)$ are fresh on the branch; $w\Rmsf^+\!u$ and $w\Rmsf^-\!u$ occur on the branch.
\begin{center}
\scriptsize{\begin{align*}
\&_\triangleright:\dfrac{w\!:\!\mathbf{1}\!:\!\phi\&\chi\triangleright\mathfrak{X}}{\begin{matrix}w\!:\!\mathbf{1}\!:\!\phi\triangleright\mathfrak{X}\\w\!:\!\mathbf{1}\!:\!\chi\triangleright\mathfrak{X}\end{matrix}}
&&
\sqcap^2_\triangleright:\dfrac{w\!:\!\mathbf{2}\!:\!\phi\!\sqcap\!\chi\triangleright\mathfrak{X}}{\begin{matrix}w\!:\!\mathbf{2}\!:\!\phi\!\triangleright\!\mathfrak{X}\\w\!:\!\mathbf{2}\!:\!\chi\!\triangleright\!\mathfrak{X}\end{matrix}}
&&
\wedge^2_\triangleright:\dfrac{w\!:\!\mathbf{2}\!:\!\phi\!\wedge\!\chi\triangleright\mathfrak{X}}{w\!:\!\mathbf{2}\!:\!\phi\!\triangleright\!\mathfrak{X}\mid w\!:\!\mathbf{2}\!:\!\chi\!\triangleright\!\mathfrak{X}}\\[.5em]
\&_\triangleleft:\dfrac{w\!:\!\mathbf{1}\!:\!\phi\&\chi\triangleleft\mathfrak{X}}{w\!:\!\mathbf{1}\!:\!\phi\!\triangleleft\!\mathfrak{X}\mid w\!:\!\mathbf{1}\!:\!\chi\!\triangleleft\!\mathfrak{X}}
&&
\wedge^2_\triangleleft:\dfrac{w\!:\!\mathbf{2}\!:\!\phi\!\wedge\!\chi\triangleleft\mathfrak{X}}{\begin{matrix}w\!:\!\mathbf{2}\!:\!\phi\!\triangleleft\!\mathfrak{X}\\w\!:\!\mathbf{2}\!:\!\chi\!\triangleleft\!\mathfrak{X}\end{matrix}}
&&
\sqcap^2_\triangleleft:\dfrac{w\!:\!\mathbf{2}\!:\!\phi\!\sqcap\!\chi\triangleleft\mathfrak{X}}{w\!:\!\mathbf{2}\!:\!\phi\!\triangleleft\!\mathfrak{X}\mid w\!:\!\mathbf{2}\!:\!\chi\!\triangleleft\!\mathfrak{X}}\\[.5em]
\Rightarrow_\triangleright:\dfrac{w\!:\!\mathbf{1}\!:\!\phi\!\Rightarrow\!\chi\triangleright\mathfrak{X}}{w\!:\!\mathbf{1}\!:\!\chi\!\triangleright\!\mathfrak{X}\left|\begin{matrix}\mathfrak{X}\!\triangleleft\!1\\w\!:\!\mathbf{1}\!:\!\chi\!\approx\!c\\w\!:\!\mathbf{1}\!:\!\phi\!\leqslant\!c\end{matrix}\right.}
&&
\rightarrow^2_>:\dfrac{w\!:\!\mathbf{2}\!:\!\phi\!\rightarrow\!\chi\!>\!\mathfrak{X}}{\begin{matrix}w\!:\!\mathbf{2}\!:\!\chi\!>\!\mathfrak{X}\\w\!:\!\mathbf{2}\!:\!\phi\!\approx\!c\\w\!:\!\mathbf{2}\!:\!\chi\!>\!c\end{matrix}}&&
\rightarrow^2_\geqslant:\dfrac{w\!:\!\mathbf{2}\!:\!\phi\!\rightarrow\!\chi\!\geqslant\!\mathfrak{X}}{\mathfrak{X}\!\leqslant\!0\left|\begin{matrix}w\!:\!\mathbf{2}\!:\!\chi\!\geqslant\!\mathfrak{X}\\w\!:\!\mathbf{2}\!:\!\phi\!\approx\!c\\w\!:\!\mathbf{2}\!:\!\chi\!>\!c\end{matrix}\right.}\\[.5em]
\sqsupset^2_\triangleright:\dfrac{w\!:\!\mathbf{2}\!:\!\phi\!\sqsupset\!\chi\triangleright\mathfrak{X}}{w\!:\!\mathbf{2}\!:\!\chi\!\triangleright\!\mathfrak{X}\left|\begin{matrix}\mathfrak{X}\!\triangleleft\!1\\w\!:\!\mathbf{2}\!:\!\chi\!\approx\!c\\w\!:\!\mathbf{1}\!:\!\phi\!\leqslant\!c\end{matrix}\right.}
&&
\Rightarrow_\triangleleft:\dfrac{w\!:\!\mathbf{1}\!:\!\phi\!\Rightarrow\!\chi\!\triangleleft\!\mathfrak{X}}{1\!\triangleleft\!\mathfrak{X}\left|\begin{matrix}w\!:\!\mathbf{1}\!:\!\chi\!\triangleleft\!\mathfrak{X}\\w\!:\!\mathbf{1}\!:\!\phi\!\approx\!c\\w\!:\!\mathbf{1}\!:\!\chi\!<\!c\end{matrix}\right.}
&&
\rightarrow^2_\triangleleft:\dfrac{w\!:\!\mathbf{2}\!:\!\phi\!\rightarrow\!\chi\!\triangleleft\!\mathfrak{X}}{w\!:\!\mathbf{2}\!:\!\chi\!\triangleleft\!\mathfrak{X}\left|\begin{matrix}0\!\triangleleft\!\mathfrak{X}\\w\!:\!\mathbf{2}\!:\!\phi\!\geqslant\!c\\w\!:\!\mathbf{2}\!:\!\chi\!\approx\!c\end{matrix}\right.}
\end{align*}}
\scriptsize{
\begin{align*}
\sqsupset^2_\triangleleft:\dfrac{w\!:\!\mathbf{2}\!:\!\phi\!\sqsupset\!\chi\!\triangleleft\!\mathfrak{X}}{1\!\triangleleft\!\mathfrak{X}\left|\begin{matrix}w\!:\!\mathbf{2}\!:\!\chi\!\triangleleft\!\mathfrak{X}\\w\!:\!\mathbf{2}\!:\!\phi\!\approx\!c\\w\!:\!\mathbf{2}\!:\!\chi\!<\!c\end{matrix}\right.}
&&
\neg:\dfrac{w\!:\!\mathbf{i}\!:\!\neg\phi\triangledown\mathfrak{X}}{w\!:\!\mathbf{j}\!:\!\phi\triangledown\mathfrak{X}}&&-:\dfrac{w\!:\!\mathbf{i}\!:\!-\!\phi\triangledown\mathfrak{X}}{w\!:\!\mathbf{j}\!:\!\phi\blacktriangledown1\!-\!\mathfrak{X}}&&\invol:\dfrac{w\!:\!\mathbf{i}\!:\!\invol\phi\triangledown\mathfrak{X}}{w\!:\!\mathbf{i}\!:\!\phi\blacktriangledown1\!-\!\mathfrak{X}}\
\end{align*}}
\scriptsize{\begin{align*}
\xcancel{\Box}_\triangleright\!\dfrac{w\!:\!\mathbf{1}\!:\!\xcancel{\Box}\phi\!\triangleright\!\mathfrak{X}}{\begin{matrix}w\!:\!\mathbf{1}\!:\!\xcancel{\Box}\phi\!=\!1\\1\!\triangleright\!\mathfrak{X}\end{matrix}\left|\begin{matrix}w\!:\!\mathbf{1}\!:\!\xcancel{\Box}\phi\!=\!t(w)\\\mathfrak{X}\!\triangleleft\!t(w)\\w'\!:\!\mathbf{1}\!:\!\phi\!<\!w\Rmsf^+\!w'\\w'\!:\!\mathbf{1}\!:\!\phi\!<\!t_1(w)\end{matrix}\right.}
&&
\Box^2_\triangleright\!\dfrac{w\!:\!\mathbf{2}\!:\!\Box\phi\!\triangleright\!\mathfrak{X}}{\mathfrak{X}\!\triangleleft\!0\left|\begin{matrix}\mathfrak{X}\!\triangleleft\!t_1(w)\\t(w)\!<\!w\Rmsf^+\!w'\\w\!:\!\mathbf{2}\!:\!\phi\!>\!t(w)\end{matrix}\right.}
&&
\blacksquare^2_\triangleright\!\dfrac{w\!:\!\mathbf{2}\!:\!\blacksquare\phi\!\triangleright\!\mathfrak{X}}{\begin{matrix}w\!:\!\mathbf{2}\!:\!\blacksquare\phi\!=\!1\\1\!\triangleright\!\mathfrak{X}\end{matrix}\left|\begin{matrix}w\!:\!\mathbf{2}\!:\!\blacksquare\phi\!=\!t(w)\\\mathfrak{X}\!\triangleleft\!t(w)\\w'\!:\!\mathbf{2}\!:\!\phi\!<\!w\Rmsf^-\!w'\\w'\!:\!\mathbf{2}\!:\!\phi\!<\!t_1(w)\end{matrix}\right.}
\end{align*}}
\scriptsize{\begin{align*}
\xcancel{\Box}_\triangleleft\!\dfrac{w\!:\!\mathbf{1}\!:\!\xcancel{\Box}\phi\!\triangleleft\!\mathfrak{X}}{\mathfrak{X}\!\triangleright\!1\left|\begin{matrix}\mathfrak{X}\!\triangleleft\!t(w)\\w'\!:\!\mathbf{1}\!:\!\phi\!<\!w\Rmsf^+w'\\w'\!:\!\mathbf{1}\!:\!\phi\!<\!t_1(w)\end{matrix}\right.}
&&
\blacksquare^2_\triangleleft\!\dfrac{w\!:\!\mathbf{2}\!:\!\blacksquare\phi\!\triangleleft\!\mathfrak{X}}{\mathfrak{X}\!\triangleright\!1\left|\begin{matrix}\mathfrak{X}\!\triangleleft\!t(w)\\w'\!:\!\mathbf{2}\!:\!\phi\!<\!w\Rmsf^-w'\\w'\!:\!\mathbf{2}\!:\!\phi\!<\!t_1(w)\end{matrix}\right.}
&&
\lozenge^2_\triangleleft\!\dfrac{w\!:\!\mathbf{2}\!:\!\lozenge\phi\!\triangleleft\!\mathfrak{X}}{\mathfrak{X}\!\triangleright\!1\left|\begin{matrix}\mathfrak{X}\!\triangleleft\!t(w)\\w'\!:\!\mathbf{2}\!:\!\phi\!<\!w\Rmsf^-w'\\w'\!:\!\mathbf{2}\!:\!\phi\!<\!t_1(w)\end{matrix}\right.}
\end{align*}}
\scriptsize{
\begin{align*}
\xcancel{\lozenge}_\triangleright\!\dfrac{w\!:\!\mathbf{1}\!:\!\xcancel{\lozenge}\phi\!\triangleright\!\mathfrak{X}}{\mathfrak{X}\!\triangleleft\!0\left|\begin{matrix}\mathfrak{X}\!\triangleleft\!t_1(w)\\t(w)\!<\!w\Rmsf^+w'\\w\!:\!\mathbf{1}\!:\!\phi\!>\!t(w)\end{matrix}\right.}
&&
\xcancel{\lozenge}_\triangleleft\!\dfrac{w\!:\!\mathbf{1}\!:\!\xcancel{\lozenge}\phi\!\triangleleft\!\mathfrak{X}}{\begin{matrix}w\!:\!\mathbf{1}\!:\!\xcancel{\lozenge}\phi\!=\!0\\\mathfrak{X}\!\triangleright\!0\end{matrix}\left|\begin{matrix}w\!:\!\mathbf{1}\!:\!\xcancel{\lozenge}\phi\!=\!t_1(w)\\t_1(w)\!\triangleleft\!\mathfrak{X}\\w\Rmsf^+w'\!>\!t(w)\\w'\!:\!\mathbf{1}\!:\!\phi\!>\!t(w)\end{matrix}\right.}
&&
\blacklozenge^2_\triangleleft\!\dfrac{w\!:\!\mathbf{2}\!:\!\blacklozenge\phi\!\triangleleft\!\mathfrak{X}}{\begin{matrix}w\!:\!\mathbf{2}\!:\!\xcancel{\lozenge}\phi\!=\!0\\\mathfrak{X}\!\triangleright\!0\end{matrix}\left|\begin{matrix}w\!:\!\mathbf{2}\!:\!\blacklozenge\phi\!=\!t_1(w)\\t_1(w)\!\triangleleft\!\mathfrak{X}\\w\Rmsf^-w'\!>\!t(w)\\w'\!:\!\mathbf{2}\!:\!\phi\!>\!t(w)\end{matrix}\right.}\end{align*}}
\scriptsize{\begin{align*}
\Box^2_\triangleleft\!\dfrac{w\!:\!\mathbf{2}\!:\!\Box\phi\!\triangleleft\!\mathfrak{X}}{\begin{matrix}w\!:\!\mathbf{2}\!:\!\Box\phi\!=\!0\\\mathfrak{X}\!\triangleright\!0\end{matrix}\left|\begin{matrix}w\!:\!\mathbf{2}\!:\!\Box\phi\!=\!t_1(w)\\t_1(w)\!\triangleleft\!\mathfrak{X}\\w\Rmsf^-w'\!>\!t(w)\\w'\!:\!\mathbf{2}\!:\!\phi\!>\!t(w)\end{matrix}\right.}
&&
\lozenge^2_\triangleright\!\dfrac{w\!:\!\mathbf{2}\!:\!\lozenge\phi\!\triangleright\!\mathfrak{X}}{\begin{matrix}w\!:\!\mathbf{2}\!:\!\lozenge\phi\!=\!1\\1\!\triangleright\!\mathfrak{X}\end{matrix}\left|\begin{matrix}w\!:\!\mathbf{2}\!:\!\lozenge\phi\!=\!t(w)\\\mathfrak{X}\!\triangleleft\!t(w)\\w'\!:\!\mathbf{2}\!:\!\phi\!<\!w\Rmsf^-w'\\w'\!:\!\mathbf{2}\!:\!\phi\!<\!t_1(w)\end{matrix}\right.}
&&
\blacklozenge^2_\triangleright\!\dfrac{w\!:\!\mathbf{2}\!:\!\blacklozenge\phi\!\triangleright\!\mathfrak{X}}{\mathfrak{X}\!\triangleleft\!0\left|\begin{matrix}\mathfrak{X}\!\triangleleft\!t_1(w)\\t(w)\!<\!w\Rmsf^-\!w'\\w\!:\!\mathbf{2}\!:\!\phi\!>\!t(w)\end{matrix}\right.}
\end{align*}}
\scriptsize{
\begin{align*}
\xcancel{\Box}_=\!\dfrac{w\!:\!\mathbf{1}\!:\!\Box\phi\!=\!\mathfrak{X}}{u\!:\!\mathbf{1}\!:\!\phi\!\geqslant\!\mathfrak{X}\left|\begin{matrix}u\!:\!\mathbf{1}\!:\!\phi\!<\!\mathfrak{X}\\w\Rmsf^+u\!\leqslant\!u\!:\!\mathbf{1}\!:\!\phi\end{matrix}\right.}
&&
\Box^2_=\!\dfrac{w\!:\!\mathbf{2}\!:\!\lozenge\phi\!=\!\mathfrak{X}}{w\Rmsf^-u\!\leqslant\!\mathfrak{X}\left|\begin{matrix}u\!:\!\mathbf{2}\!:\!\phi\!\leqslant\!\mathfrak{X}\\w\Rmsf^-\!u\!>\mathfrak{X}\end{matrix}\right.}
&&
\blacksquare^2_=\!\dfrac{w\!:\!\mathbf{2}\!:\!\blacksquare\phi\!=\!\mathfrak{X}}{u\!:\!\mathbf{2}\!:\!\phi\!\geqslant\!\mathfrak{X}\left|\begin{matrix}u\!:\!\mathbf{2}\!:\!\phi\!<\!\mathfrak{X}\\w\Rmsf^-\!u\!\leqslant\!u\!:\!\mathbf{2}\!:\!\phi\end{matrix}\right.}\\[.5em]
\xcancel{\lozenge}_=\!\dfrac{w\!:\!\mathbf{1}\!:\!\xcancel{\lozenge}\phi\!=\!\mathfrak{X}}{w\Rmsf^+u\!\leqslant\!\mathfrak{X}\left|\begin{matrix}u\!:\!\mathbf{1}\!:\!\phi\!\leqslant\!\mathfrak{X}\\w\Rmsf^+\!u\!>\mathfrak{X}\end{matrix}\right.}
&&
\blacklozenge^2_=\!\dfrac{w\!:\!\mathbf{2}\!:\!\blacklozenge\phi\!=\!\mathfrak{X}}{w\Rmsf^+u\!\leqslant\!\mathfrak{X}\left|\begin{matrix}u\!:\!\mathbf{2}\!:\!\phi\!\leqslant\!\mathfrak{X}\\w\Rmsf^-\!u\!>\mathfrak{X}\end{matrix}\right.}
&&
\lozenge^2_=\!\dfrac{w\!:\!\mathbf{2}\!:\!\lozenge\phi\!=\!\mathfrak{X}}{u\!:\!\mathbf{2}\!:\!\phi\!\geqslant\!\mathfrak{X}\left|\begin{matrix}u\!:\!\mathbf{2}\!:\!\phi\!<\!\mathfrak{X}\\w\Rmsf^-\!u\!\leqslant\!u\!:\!\mathbf{2}\!:\!\phi\end{matrix}\right.}\\[.5em]
\xcancel{\Box}_{\!\approx}\!\dfrac{w\!:\!\mathbf{1}\!:\!\xcancel{\Box}\phi\!\approx\!\mathfrak{X}}{\begin{matrix}w\!:\!\mathbf{1}\!:\!\xcancel{\Box}\phi\!=\!1\\\mathfrak{X}\!\!=\!\!1\end{matrix}\left|\begin{matrix}w\!:\!\mathbf{1}\!:\!\xcancel{\Box}\phi\!\!=\!\!t(w)\\\mathfrak{X}\!=\!t(w)\\w'\!:\!\mathbf{1}\!:\!\phi\!\!<\!\!w\Rmsf^+\!w'\\w'\!:\!\mathbf{1}\!:\!\phi\!\!<\!\!t_1(w)\end{matrix}\right.}
&&
\Box^2_{\!\approx}\!\dfrac{w\!:\!\mathbf{2}\!:\!\Box\phi\!\approx\!\mathfrak{X}}{\begin{matrix}w\!:\!\mathbf{2}\!:\!\Box\phi\!\!=\!\!0\\\mathfrak{X}\!=\!0\end{matrix}\left|\begin{matrix}w\!:\!\mathbf{2}\!:\!\Box\phi\!\!=\!\!t_1(w)\\w\Rmsf^-\!w'\!>\!t(w)\\w'\!:\!\mathbf{2}\!:\!\phi\!>\!t(w)\\\mathfrak{X}\!=\!t_1(w)\end{matrix}\right.}
&&
\blacksquare^2_{\!\approx}\!\dfrac{w\!:\!\mathbf{2}\!:\!\xcancel{\Box}\phi\!\approx\!\mathfrak{X}}{\begin{matrix}w\!\!:\!\mathbf{2}\!\!:\!\blacksquare\phi\!\!=\!\!1\\\mathfrak{X}\!\!=\!\!1\end{matrix}\left|\begin{matrix}w\!:\!\mathbf{2}\!:\!\blacksquare\phi\!\!=\!\!t(w)\\\mathfrak{X}\!=\!t(w)\\w'\!\!:\!\mathbf{2}\!\!:\!\phi\!\!<\!\!w\Rmsf^-\!\!w'\\w'\!:\!\mathbf{2}\!:\!\phi\!\!<\!\!t_1(w)\end{matrix}\right.}
\\[.5em]
\xcancel{\lozenge}_{\!\approx}\!\dfrac{w\!:\!\mathbf{1}\!:\!\xcancel{\lozenge}\phi\!\approx\!\mathfrak{X}}{\begin{matrix}w\!:\!\mathbf{1}\!:\!\xcancel{\lozenge}\phi\!\!=\!\!0\\\mathfrak{X}\!=\!0\end{matrix}\left|\begin{matrix}w\!:\!\mathbf{1}\!:\!\xcancel{\lozenge}\phi\!\!=\!\!t_1(w)\\w\Rmsf^+\!w'\!>\!t(w)\\w'\!:\!\mathbf{1}\!:\!\phi\!>\!t(w)\\\mathfrak{X}\!=\!t_1(w)\end{matrix}\right.}
&&
\lozenge^2_{\!\approx}\!\dfrac{w\!:\!\mathbf{2}\!:\!\lozenge\phi\!\approx\!\mathfrak{X}}{\begin{matrix}w\!\!:\!\mathbf{2}\!\!:\!\lozenge\phi\!\!=\!\!1\\\mathfrak{X}\!\!=\!\!1\end{matrix}\left|\begin{matrix}w\!:\!\mathbf{2}\!:\!\lozenge\phi\!\!=\!\!t(w)\\\mathfrak{X}\!=\!t(w)\\w'\!\!:\!\mathbf{2}\!\!:\!\phi\!\!<\!\!w\Rmsf^-\!\!w'\\w'\!:\!\mathbf{2}\!:\!\phi\!\!<\!\!t_1(w)\end{matrix}\right.}
&&
\blacklozenge^2_{\!\approx}\!\dfrac{w\!:\!\mathbf{2}\!:\!\blacklozenge\phi\!\approx\!\mathfrak{X}}{\begin{matrix}w\!:\!\mathbf{2}\!:\!\blacklozenge\phi\!\!=\!\!0\\\mathfrak{X}\!\!=\!\!0\end{matrix}\left|\begin{matrix}w\!\!:\!\mathbf{2}\!\!:\!\blacklozenge\phi\!\!=\!\!t_1(w)\\w\Rmsf^-\!w'\!>\!t(w)\\w'\!:\!\mathbf{2}\!:\!\phi\!>\!t(w)\\\mathfrak{X}\!=\!t_1(w)\end{matrix}\right.}
\end{align*}}
\end{center}

Let $\mathcal{B}=\{\mathcal{c}_1,\ldots,\mathcal{c}_n\}$ be a branch with constraints $\mathcal{c}_1$, \ldots, $\mathcal{c}_n$. We translate (parts of) constraints as follows (below, $\Smsf\in\{\Rmsf^+,\Rmsf^-\}$).
\begin{align}\label{equ:translation}
(w\!:\!\mathbf{i}\!:\!\phi)^\mathcal{t}&=x_{w,i,\phi}&(w\Smsf w')^\mathcal{t}&=x_{w\Smsf w'}&c^\mathcal{t}&=c&(t(w))^\mathcal{t}&=x_{t(w)}\nonumber\\
0^\mathcal{t}&=0&1^\mathcal{t}&=1&(\mathfrak{X}\triangledown\mathfrak{X}')^\mathcal{t}&=\mathfrak{X}^\mathcal{t}\triangledown{\mathfrak{X}'}^\mathcal{t}
\end{align}

$\mathcal{B}$ is \emph{closed} if the following system of inequalities
\begin{align}\label{equ:closure}
\mathcal{c}^\mathcal{t}_1,\ldots,\mathcal{c}^\mathcal{t}_n\\
(t^i_0(w))^\mathcal{t}<(t^i_1(w))^\mathcal{t}\tag{for every $t^i_j(w)$ occurring in $\mathcal{B}$}
\end{align}
\emph{has no solution over $[0,1]$} s.t.\ $\neg\exists x_{t(w)},x_{t^i_0(w)},x_{t^i_1(w)}:x_{t^i_0(w)}<x_{t(w)}<x_{t^i_1(w)}$.

An \emph{open} (i.e., non-closed) branch $\mathcal{B}$ is \emph{complete} when for every premise of any rule occurring on $\mathcal{B}$, its conclusion also occurs on $\mathcal{B}$.

Finally, $\phi\in\fullLbilattice$ \emph{has a~$\TKblG$ proof} if there are 
\emph{two tableaux} beginning with $\{w:\mathbf{1}:\phi<c,c<1\}$ and $\{w:\mathbf{2}:\phi>d,d>0\}$ s.t.\ all their branches are closed. Similarly, $\chi\in\bimodalLinv$ \emph{has a~$\TKblG$ proof} if there is a~tableaux beginning with $\{w:\mathbf{1}:\chi<c,c<1\}$ s.t.\ all its branches are closed.
\end{definition}
\begin{figure}
\centering
\fontsize{6.5}{7.25}\selectfont{\begin{forest}
smullyan tableaux
[{w\!\!:\!\!\Box p\!\!\rightarrow\!\!\invol\lozenge\invol p\scalebox{.7}[1]{$\leqslant$}c}
[{c\scalebox{.7}[1]{$<$}1}
[{w\!\!:\!\!\Box p\scalebox{.7}[1]{$>$}c'}
[{c'\scalebox{.7}[1]{$\leqslant$}c}
[{w\!\!:\!\!\lozenge\invol p\scalebox{.7}[1]{$\approx$}1\scalebox{.7}[1]{-}c'}
[w\!\!:\!\!{{\Box p\scalebox{.7}{$=$}1}}
[{w\!\!:\!\!\lozenge\invol p\scalebox{.7}{$=$}t_1}[w\Rmsf u\!\!>\!\!t_0[{u\!\!:\!\!p\scalebox{.7}{$<$}1\scalebox{.7}[1]{-}t_0}
[{w\Rmsf u\scalebox{.7}{$\leqslant$}t_1}[{u\!\!:\!\!p\scalebox{.7}{$\geqslant$}1}[\times]][{u\!\!:\!\!p\scalebox{.7}{$<$}1}[{u\!\!:\!\!p\scalebox{.7}{$\geqslant$}w\Rmsf u}
[\circledast_A]
]]]
[{w\Rmsf u\scalebox{.7}{$>$}t_1}[{u\!\!:\!\!p\scalebox{.7}{$\geqslant$}1\scalebox{.7}[1]{-}t_1}
[{u\!\!:\!\!p\scalebox{.7}{$\geqslant$}1}[\times]][{u\!\!:\!\!p\scalebox{.7}{$<$}1}[{u\!\!:\!\!p\scalebox{.7}{$\geqslant$}w\Rmsf u}
[\circledast]
]]]]]]]
[{w\!\!:\!\!\lozenge\invol p\scalebox{.7}{$\approx$}0}[{1\scalebox{.7}[1]{-}c'\scalebox{.7}{$=$}0}[\times]]]]
[{w\!\!:\!\!\Box p\scalebox{.7}{$=$}t_0}
[{c'\scalebox{.7}{$<$}t_0}[{u\!\!:\!\!p\scalebox{.7}{$<$}t_1}[{u\!\!:\!\!p\scalebox{.7}{$<$}w\Rmsf u}
[{w\!\!:\!\!\lozenge\invol p\scalebox{.7}{$\approx$}0}[{1\scalebox{.7}[1]{-}c'\scalebox{.7}{$=$}0}[\times]]][{w\!\!:\!\!\lozenge\invol p\scalebox{.7}{$=$}t'_1}[{w\Rmsf x\scalebox{.7}{$>$}t'_0}[{x\!\!:\!\!p\scalebox{.7}{$<$}1\scalebox{.7}[1]{-}t'_0}
[{w\Rmsf x\scalebox{.7}{$\leqslant$}t'_1}[{u\!\!:\!\!p\scalebox{.7}{$\geqslant$}t_0}
[{x\!\!:\!\!p\scalebox{.7}{$\geqslant$}t_0}[{w\Rmsf u\scalebox{.7}{$\leqslant$}t'_1}[\circledast]][{w\Rmsf u\scalebox{.7}{$>$}t'_1}[{u\!\!:\!\!p\scalebox{.7}{$\geqslant$}1\scalebox{.7}[1]{-}t'_1}[\circledast]]]][{x\!\!:\!\!p\scalebox{.7}{$<$}t_0}[{x\!\!:\!\!p\scalebox{.7}{$\geqslant$}w\Rmsf x},s sep=0[{w\Rmsf u\scalebox{.7}{$\leqslant$}t'_1}[\circledast]][{w\Rmsf u\scalebox{.7}{$>$}t'_1}[{u\!\!:\!\!p\scalebox{.7}{$\geqslant$}1\scalebox{.7}[1]{-}t'_1}[\circledast]]]]]
]
[{u\!\!:\!\!p\scalebox{.7}{$<$}t_0}[{u\!\!:\!\!p\scalebox{.7}{$\geqslant$}w\Rmsf u}[\times]]]][{w\Rmsf x\scalebox{.7}{$>$}t'_1}[{x\!\!:\!\!p\scalebox{.7}{$\geqslant$}1\scalebox{.7}[1]{-}t'_1}[u\!\!:\!\!p\scalebox{.7}{$\geqslant$}t_0
[{x\!\!:\!\!p\scalebox{.7}{$\geqslant$}t_0}[{w\Rmsf u\scalebox{.7}{$\leqslant$}t'_1}[\circledast]][{w\Rmsf u\scalebox{.7}{$>$}t'_1}[{u\!\!:\!\!p\scalebox{.7}{$\geqslant$}1\scalebox{.7}[1]{-}t'_1}[\circledast]]]][{x\!\!:\!\!p\scalebox{.7}{$<$}t_0}[{x\!\!:\!\!p\scalebox{.7}{$\geqslant$}w\Rmsf x}[{w\Rmsf u\scalebox{.7}{$\leqslant$}t'_1}[\circledast]][{w\Rmsf u\scalebox{.7}{$>$}t'_1}[{u\!\!:\!\!p\scalebox{.7}{$\geqslant$}1\scalebox{.7}[1]{-}t'_1}[\circledast_B]]]]]
]
[{u\!\!:\!\!p\scalebox{.7}{$<$}t_0}[{u\!\!:\!\!p\scalebox{.7}{$\geqslant$}w\Rmsf u}[\times]]]]]]]]]]]]]]]]]
\end{forest}}
\caption{A failed proof of $\Box p\!\rightarrow\!\invol\lozenge\invol p$ on fuzzy frames. $\circledast$'s mark open branches.}
\label{fig:proofexample}
\end{figure}
\begin{definition}[Model realising a~branch]\label{def:realisingmodel}
$\mathcal{B}$ is \emph{realised by a~mo\-del} $\mathfrak{M}=\langle W,R^+,R^-,T_1,T_2,v_1,v_2\rangle$ with $W=\{w:w\text{ occurs on }\mathcal{B}\}$ iff there is is a~map $\real:\Str\rightarrow[0,1]$ s.t.\ for every structure occurring in $\mathcal{B}$:
\begin{itemize}[noitemsep,topsep=2pt]
\item $\real(0)\!=\!0$, $\real(1)\!=\!1$, and $\real(\mathfrak{X})\triangledown\real(\mathfrak{X}')$ if $\mathfrak{X}\triangledown\mathfrak{X}'\!\in\!\mathcal{B}$ (if $\mathfrak{X}$ is not on $\mathcal{B}$, $\real(\mathfrak{X})\!=\!0$);
\item $\real(w:\mathbf{i}:\phi)=v_i(\phi,w)$, $\real(w\Rmsf^+w')=wR^+w'$, and $\real(w\Rmsf^-w')=wR^-w'$;
\item if $\real(t)=x$, then $\real(1-t)=1-x$ for every value term $t$;
\item $T_i(w)=\{\real(t(w))\mid u:\mathbf{i}:\phi\triangledown t(w)\in\mathcal{B}\}\cup\{0,1\}$;
\item there are no $\real(t'(w)),\real(t_0(w)),\real(t_1(w))\!\!\in\!\!T_i$ s.t.\ $\real(t_0(w))\!\!<\!\!\real(t'(w))\!\!<\!\!\real(t_1(w))$.
\end{itemize}
\end{definition}
\begin{example}[Failed tableau proof and realising models]\label{example:tableau}
We show a failed proof of $\Box p\rightarrow\invol\lozenge\invol p$ in $\TKblG$. For brevity, we treat it as an~$\bimodalLinv$ formula: we write $\Rmsf$ instead of $\Rmsf^+$ and omit $\mathbf{1}$ in constraints. Additionally, we skip the applications of the $\invol$ rule and write $t_0$, $t'_1$, etc.\ instead of $t^0_0(w)$, $t^1_1(w)$,\ldots since we evaluate modal formulas in one state only. The proof is in Fig.~\ref{fig:proofexample}.

Consider now open branches $\circledast_A$ and $\circledast_B$. For each of them, there are uncountably many realising models corresponding to solutions of $(\circledast_A)^\mathcal{t}$ and $(\circledast_B)^\mathcal{t}$ (cf.~\eqref{equ:translation} and~\eqref{equ:closure}). Examples of models are shown in Fig.~\ref{fig:leftmodel} and~\ref{fig:rightmodel}. One can see that $v(\Box p,w)=1$ and $v(\invol\lozenge\invol p,w)=\frac{3}{4}$ in the left model and $v(\Box p,w)=\frac{9}{10}$ and $v(\invol\lozenge\invol p,w)=\frac{8}{9}$ in the right model.
\begin{figure}
\begin{minipage}{.44\textwidth}
\centering
\small{\begin{align*}
\xymatrix{w~\ar[rr]|(.3){\sfrac{1}{5}}&&~u\!:\!p\!=\!\sfrac{2}{5}}\\[.3em]
T(w)\!=\!\left\{0,\sfrac{1}{6},\sfrac{1}{4},1\right\}\\[.3em]
\real(t_0)\!=\!\sfrac{1}{6},\real(t_1)\!=\!\sfrac{1}{4}
\end{align*}}
\caption{Realising model for~$\circledast_A$.}
\label{fig:leftmodel}
\end{minipage}
\hfill
\begin{minipage}{.54\textwidth}
\centering
\small{\begin{align*}
\xymatrix{x\!:\!p\!=\!\sfrac{8}{9}~&&w~\ar[ll]|(.3){\sfrac{1}{4}}\ar[rr]|(.3){\sfrac{13}{14}}&&~u\!:\!p\!=\!\sfrac{11}{12}}\\[.3em]
T(w)=\left\{0,\sfrac{1}{11},\sfrac{1}{9},\sfrac{9}{10},\sfrac{12}{13},1\right\}\\[.3em]
\real(t'_0)\!=\!\sfrac{1}{11},\real(t'_1)\!=\!\sfrac{1}{9},\real(t_0)\!=\!\sfrac{9}{10},\real(t_1)\!=\!\sfrac{12}{13}
\end{align*}}
\caption{Realising model for $\circledast_B$.}
\label{fig:rightmodel}
\end{minipage}
\end{figure}
\end{example}
\begin{theorem}[$\TKblG$ completeness]\label{theorem:TKblGcompleteness}~
\begin{enumerate}[noitemsep,topsep=2pt]
\item $\phi\in\fullLbilattice$ is $\KblG$-valid iff it has a~$\TKblG$ proof.
\item $\chi\in\bimodalLinv$ is $\KinvG$-valid iff it has a~$\TKblG$ proof.
\end{enumerate}
\end{theorem}
\begin{proof}
We prove only 1.\ since 2.\ can be shown in the same way. The proof is standard and follows~\cite[Theorems~6.19 and~6.21]{Rogger2016phd}. 
%
%
The soundness proof is straightforward, and we put it in the appendix (Section~\ref{ssec:theorem:TKblGcompleteness:soundness}). For completeness, we prove that every complete open branch $\mathcal{B}$ has a~realising model. We consider a~solution for $(\mathcal{B})^\mathcal{t}$ as specified in Definition~\ref{def:TKblG} and define the realising model $\mathfrak{M}$ as follows: $W=\{w\mid w\text{ occurs on }\mathcal{B}\}$ and $\real(\mathfrak{X})=\mathcal{V}(\mathfrak{X}^\mathcal{t})$ with $\mathcal{V}(\mathfrak{X}^\mathcal{t})$ being the value of $\mathfrak{X}^\mathcal{t}$ in the solution of $(\mathcal{B})^\mathcal{t}$. Namely, $w\Smsf w'=\mathcal{V}(x_{w\Smsf w'})$, $v_i(\phi,w)=\mathcal{V}(x_{w,i,\phi})$, $T_i(w)=\{\mathcal{V}(x_{t(w)})\mid u\!:\!\mathbf{i}\!:\!\phi\triangledown t(w)\in\mathcal{B}\}\cup\{0,1\}$.

It remains to show that all formulas on $\mathcal{B}$ are realised. Constraints of the form $w\!:\!\mathbf{i}\!:\!p\triangledown\mathfrak{X}$ are realised by construction. For complex formulas, we proceed by induction. The cases of propositional connectives can be shown by simple applications of the induction hypothesis. Modalities can all be tackled similarly.

Assume that $w\!:\!\mathbf{1}\!:\!\Box\psi\approx\mathfrak{X}\!\in\!\mathcal{B}$. Since $\mathcal{B}$ is complete, we have two cases (cf.~rule $\xcancel{\Box}_\approx$). First, $w:\mathbf{1}:\Box\psi=1$ and \underline{$\mathfrak{X}=1$} are also in $\mathcal{B}$. Thus, by rule $\xcancel{\Box}_=$, for every $w'$ s.t.\ $w\Rmsf^+\!w'$ is on $\mathcal{B}$\footnote{Recall that $\real(w\Rmsf^+w')=0$ if $w\Rmsf^+\!w'$ is \emph{not} on $\mathcal{B}$. Thus, the value of $\psi$ there does not affect the value of $\Box\psi$.}, either $\underline{w':\mathbf{1}:\psi\geqslant1}\in\mathcal{B}$ or $\underline{w':\mathbf{1}:\psi<1}\in\mathcal{B}$ and $\underline{w':\mathbf{1}:\psi\geqslant w\Rmsf^+w'}\in\mathcal{B}$. By the induction hypothesis, the underlined constraints are realised. But then $v_1(\Box\psi,w)=1$.

In the second case, $w\!:\!\mathbf{1}\!:\!\Box\psi\!=\!t(w)$ and $\underline{\mathfrak{X}\!=\!t(w)}$ are on $\mathcal{B}$, and there are fresh $w'$ and $t(w)\!\in\!\Tmsf$ s.t.\ $w\!:\!\mathbf{1}\!:\!\Box\psi=t(w)$, $\underline{w'\!:\!\mathbf{1}\!:\!\psi<w\Rmsf^+\!w'}$, and $\underline{w'\!:\!\mathbf{1}\!:\!\phi<t_1(w)}$ belong to $\mathcal{B}$.  Furthermore, by $\xcancel{\Box}_=$, for each $w'$ s.t.\ $w\Rmsf^+\!w'$ is on $\mathcal{B}$, $\underline{w'\!:\!\mathbf{1}\!:\!\psi\!\geqslant\!t(w)}$ is on $\mathcal{B}$ or both $\underline{w':\mathbf{1}:\psi<t(w)}$ and $\underline{w'\!:\!\mathbf{1}\!:\!\psi\geqslant w\Rmsf^+\!w'}$ are on $\mathcal{B}$. By the induction hypothesis, the underlined constraints are realised. I.e., $w\Rmsf^+w'\!>\!v(\psi,w')\!\geqslant\!\real(t(w))$, and $\real(t_1(w))\!>\!v(\psi,w')$ and $v(\psi,w'')\!\geqslant\!t(w)$ in every $w''$ s.t.\ $w\Rmsf^+\!w''\!>\!v(\psi,w'')$. Hence, $v_1(\Box\psi,w)\!=\!t(w)\!=\!\real(\mathfrak{X})$.
\end{proof}

\begin{theorem}\label{theorem:pspacecompleteness}
Validity of $\KblG$ and $\KinvG$ is $\pspace$-complete.
\end{theorem}
\begin{proof}
$\pspace$-hardness is simple. Recall that $\KG$ is $\pspace$-hard~\cite{CaicedoMetcalfeRodriguezRogger2017} and $\KinvG$ conservatively expands $\KG$. For $\KblG$, Theorem~\ref{theorem:embedding} gives us the reduction of $\KinvG$-validity to $\KblG$-validity. For $\pspace$-membership, we adapt the standard algorithm for $\mathbf{K}$ from~\cite{BlackburndeRijkeVenema2010}.

To show that $\phi$ is valid, begin the tableau with $\{w_0:\mathbf{1}:\phi\leqslant c,c<1\}$.\footnote{The tableau for $\{w_0:\mathbf{2}:\phi>d,d>0\}$ can be built in the same way.} First, apply propositional rules beginning with non-branching ones. If a~rule requires branching, we guess one branch of the tableau and work with it until we can close it (in which case, we go for the next branch) or until it is complete and open (in which case, $\phi$ is not valid). If all branches \emph{in both tableaux} are closed, $\phi$ is valid. Once all propositional rules are applied, we guess the realising model as given in Definitions~\ref{def:TKblG} and~\ref{def:realisingmodel}. Note that \emph{verifying the solution} takes polynomial time w.r.t.\ $|\phi|$ since we are decomposing it. If there is no solution, the branch is closed. Else, we pick one modal formula, say $w\!:\!\mathbf{1}\!:\!\Box\chi\!\leqslant\!\mathfrak{X}$ and apply the rule (in our case, $\xcancel{\Box}_\triangleleft$) and guess the branch. If a new state (say, $w^1_1$) was introduced, apply propositional rules to $\chi$ in $w^1_1$ and check whether the branch became closed. If the branch is still open, pick the next modal formula and repeat the process with all modal formulas in $w_0$ guessing the branches when needed. This will add new states $w^1_1$, \ldots, $w^1_{k_1}$ (we call them \emph{successors of $w_0$}) with $k_1\leq|\phi|$ for each of which we need $O(|\phi|)$ space to store immediate subformulas of $\phi$, a~relational term $w_0\Rmsf^+w^1_i$ or $w_0\Rmsf^-w^1_i$, and $t(w_0)$'s. After checking that the branch is not closed, apply modal rules that use the introduced states. Then, apply propositional rules in each introduced state and check whether the branch is closed.

If the branch is not closed, we pick $w^1_1$ and repeat the process: pick a modal formula, e.g., $w^1_1:\mathbf{2}:\blacklozenge\psi>\mathfrak{X}'$, apply the needed modal rule and guess the branch. If a new state ($w^2_1$) is introduced, we decompose other modal formulas in $w^2_1$, generate $w^2_2$, \ldots, $w^2_{k_2}$, and then apply the rules that use these states. We repeat this until we reach a state $w^m_i$ s.t.\ all states $w^{m+1}_1,\ldots,w^{m+1}_{k_{m+1}}$ generated from it contain only propositional formulas. If the branch is not closed, we delete $w^{m+1}_1,\ldots,w^{m+1}_{k_{m+1}}$, mark $w^m_i$ safe, and proceed to $w^m_{i+1}$. Once all successors of $w^m_{i+1}$ are marked safe, it is marked safe, and we proceed to $w^m_{i+2}$, etc. If $w_0$ is marked safe, the branch of the tableau is complete and open (whence, $\phi$~is not valid). Note that the depth of the model is bounded by $|\phi|$ and that each state of the model has $O(|\phi|)$ successors. However, at each stage of the procedure, we have to store only $O(|\phi|^2)$ states each of which requires $O(|\phi|)$ space to store subformulas of $\phi$, their values, elements of $T_1$ and $T_2$, and relational terms.
\end{proof}
\section{Conclusion\label{sec:conclusion}}
We presented two expansions of the G\"{o}del modal logic: $\KinvG$ --- with the involutive negation and $\KblG$ --- with bi-lattice connectives. We showed that they can be faithfully embedded into one another and also obtained the $\pspace$-completeness of their validity via a tableaux calculus that we constructed for an alternative semantics following~\cite{CaicedoMetcalfeRodriguezRogger2013,CaicedoMetcalfeRodriguezRogger2017}.

Our next steps are as follows. First, we want to provide complete Hilbert-style axiomatisations of $\KinvG$ and $\KblG$ both over crisp and over fuzzy frames. Second, we plan a more systematic study of the correspondence theory of $\KinvG$ and $\KblG$. Third, in~\cite{MetcalfeOlivetti2009,MetcalfeOlivetti2011} hypersequent calculi for $\Box$- and $\lozenge$-fragments of $\KG$ are presented. It makes sense to come up with terminating hypersequent calculi for the full $\KG$ and its expansions constructed in this paper as well as in~\cite{BilkovaFrittellaKozhemiachenko2022IJCAR,BilkovaFrittellaKozhemiachenko2023IGPL}.

In parallel with these tasks, we intend to build off the work of \cite{BobilloDelgadoGomez-RamiroStraccia2009} to engineer G\"{o}del description logics that are able to record subtle features of roles. The foregoing work supports multi-lateral treatments of roles, allowing a description language to track the verification, falsification, and informativeness of a role's holding between two terms. Planned future work will develop these intuitions, exploring both their formalisation and potential applications.
\bibliographystyle{splncs04}
\bibliography{reference}
\newpage
\appendix
\section{Proofs of Section~\ref{sec:KinvG}}
\subsection{Theorem~\ref{theorem:Boxlozengeinterdefinability}, item~2\label{ssec:theorem:Boxlozengeinterdefinability}}
We show that $\Box$ and $\lozenge$ are not interdefinable on fuzzy frames.
\begin{proof}
Take the model below (all variables have the same values as $p$) and note that $v(\Box p,w)=\sfrac{1}{5}$ and $v(\lozenge p,w)=\sfrac{1}{4}$.
\begin{align*}
\xymatrix{w'':p=\sfrac{1}{4}~&&~\ar[ll]|{\sfrac{2}{3}}w:p=1\ar[rr]|{\sfrac{2}{3}}&&~w':p=\sfrac{1}{5}}
\end{align*}

First, we obtain by induction that
\begin{align}
\forall\tau\in\bimodalLinv:v(\tau,w')\in\left\{0,\sfrac{1}{5},\sfrac{4}{5},1\right\}\text{ and }v(\tau,w'')\in\left\{0,\sfrac{1}{4},\sfrac{3}{4},1\right\}\label{equ:ab}
\tag{a}
\end{align}
The basis cases in~\eqref{equ:ab} hold by the construction of the model and the cases of propositional connectives are obtained by the application of the induction hypothesis. Finally, as $R(w')=R(w'')=\varnothing$, for $\tau=\lozenge\sigma$, we have $v(\lozenge\sigma,w')=v(\lozenge\sigma,w'')=0$; and for $\tau=\Box\sigma$, we have that $v(\Box\sigma,w')=v(\Box\sigma,w'')=1$.

Then, we can show for every $\tau\in\bimodalLinv$ that
\begin{align}\label{equ:Boxlozengeinterdefinability}
v(\tau,w'')=0&\Rightarrow v(\tau,w')=0&v(\tau,w'')=\sfrac{1}{4}&\Rightarrow v(\tau,w')=\sfrac{1}{5}\nonumber\\
v(\tau,w'')=\sfrac{3}{4}&\Rightarrow v(\tau,w')=\sfrac{4}{5}&v(\tau,w'')=1&\Rightarrow v(\tau,w')=1
\tag{b}
\end{align}
and
\begin{align}\label{equ:Boxlozengeinterdefinabilityback}
v(\tau,w')=0&\Rightarrow v(\tau,w'')=0&v(\tau,w')=\sfrac{1}{5}&\Rightarrow v(\tau,w'')=\sfrac{1}{4}\nonumber\\
v(\tau,w')=\sfrac{4}{5}&\Rightarrow v(\tau,w'')=\sfrac{3}{4}&v(\tau,w')=1&\Rightarrow v(\tau,w'')=1
\tag{b$'$}
\end{align}
For~\eqref{equ:Boxlozengeinterdefinability}, we proceed by induction. The basis case of $\tau=p$ holds by the construction of the model. If $\tau=\lozenge\sigma$, then $v(\lozenge\sigma,w')=v(\lozenge\sigma,w'')=0$. If $\tau=\Box\sigma$, then $v(\Box\sigma,w')=v(\Box\sigma,w'')=1$.

Let $\tau=\varrho\wedge\sigma$. The cases of $v(\tau,w'')\in\{0,1\}$ are straightforward. Assume that $v(\varrho\wedge\sigma,w'')=\frac{1}{4}$, then w.l.o.g.\ $v(\varrho,w'')=\frac{1}{4}$ and $v(\sigma,w'')\in\{\frac{1}{4},\frac{3}{4},1\}$. By the induction hypothesis, we have that $v(\varrho,w'')=\frac{1}{5}$ and $v(\sigma,w'')\in\{\frac{1}{5},\frac{4}{5},1\}$, whence, $v(\varrho\wedge\sigma,w')=\frac{1}{5}$. The case when $v(\varrho\wedge\sigma,w'')=\frac{3}{4}$ is tackled in the same way. Consider $\tau=\varrho\rightarrow\sigma$. We deal only with the most instructive case of $v(\varrho\rightarrow\sigma,w'')=\frac{1}{4}$. Here, we have $v(\sigma,w'')=\frac{1}{4}$ and $v(\varrho,w'')\in\{1,\frac{3}{4}\}$. By the induction hypothesis, we obtain $v(\sigma,w')=\frac{1}{5}$ and $v(\varrho,w')\in\{\frac{4}{5},1\}$. Hence, $v(\varrho\rightarrow\sigma,w')=\frac{1}{5}$.

For~\eqref{equ:Boxlozengeinterdefinabilityback}, the case of $\tau=p$ holds by construction; if $\tau=\Box\sigma$, then $v(\Box\sigma,w')=v(\Box\sigma,w'')=1$; if $\tau=\lozenge\sigma$, then $v(\lozenge\sigma,w')=v(\lozenge\sigma,w'')=0$. The remaining cases of $\tau=\varrho\wedge\sigma$ and $\tau=\varrho\rightarrow\sigma$ can be dealt with in the same way as for~\eqref{equ:Boxlozengeinterdefinability}.

We now show by induction that (i) there is no $\Box$-free formula $\chi$ s.t.\ $v(\chi,w)\in\left\{\frac{1}{5},\frac{4}{5}\right\}$ and (ii) there is no $\lozenge$-free formula $\psi$ s.t.\ $v(\psi,w)\in\left\{\frac{1}{4},\frac{3}{4}\right\}$.

For (i), the basis case holds by the construction of the model. The cases of propositional connectives can also be established by a straightforward application of the induction hypothesis. Now let $\chi=\lozenge\tau$. It is clear that $v(\lozenge\tau,w)\neq\frac{4}{5}$ since $wRw'=wRw''=\frac{2}{3}$. We check that $v(\lozenge\tau,w)\neq\frac{1}{5}$.

Assume for contradiction that $v(\lozenge\tau,w)=\frac{1}{5}$. Then, $v(\tau,w')=\frac{1}{5}$ from~\eqref{equ:ab}. But using~\eqref{equ:Boxlozengeinterdefinability}, we have that $v(\tau,w'')=\frac{1}{4}$, whence $v(\lozenge\tau,w)=\frac{1}{4}$. Contradiction.

For (ii), the basis case holds by the construction of the model and the cases of propositional connectives can be proven by the application of the induction hypothesis. We show that $v(\Box\tau,w)\notin\left\{\frac{1}{4},\frac{3}{4}\right\}$. Observe that $wRw'=wRw''=\frac{2}{3}$, whence, $v(\Box\tau,w)\!\neq\!\frac{3}{4}$. We are going to check that $v(\Box\tau,w)\!\neq\!\frac{1}{4}$. The proof is similar to the one of (i) but we use~\eqref{equ:Boxlozengeinterdefinabilityback}. If $v(\Box\tau,w)=\frac{1}{4}$, then $v(\tau,w'')=\frac{1}{4}$ by~\eqref{equ:ab}. But from here, we obtain $v(\tau,w')=\frac{1}{5}$ by~\eqref{equ:Boxlozengeinterdefinabilityback}. Contradiction.
\end{proof}
\subsection{Theorem~\ref{theorem:framedefinability}, item~2\label{ssec:theorem:framedefinability}}
We show that $\mathfrak{F}\models_{\KinvG}\invol\triangle(\lozenge\mathbf{1}\rightarrow\invol\lozenge\mathbf{1})$ iff for every $w\in W$ there is $w'$ s.t.\ $wRw'>\frac{1}{2}$.
\begin{proof}
Let $\mathfrak{F}$ be s.t.\ for every $w\in W$ there is $w'$ with $wRw'>\frac{1}{2}$. Hence, $v(\lozenge\mathbf{1})>\frac{1}{2}$ in all $w$'s. But then, $v(\invol\triangle(\lozenge\mathbf{1}\rightarrow\invol\lozenge\mathbf{1}),w)=1$ in all $w$'s too, as required. For the converse, assume that there is some $w$ s.t.\ $wRw'\leq\frac{1}{2}$ for all $w'\in W$. From here, $\sup\{wRw':w'\in W\}\leq\frac{1}{2}$, whence $v(\lozenge\mathbf{1},w)\leq\frac{1}{2}$, and thus, $v(\invol\triangle(\lozenge\mathbf{1}\rightarrow\invol\lozenge\mathbf{1}),w)=0$, as required.
\end{proof}
\subsection{Theorem~\ref{theorem:framedefinability}, item~3\label{ssec:theorem:framedefinability3}}
Let $0<x,x'<1$, and consider two frames: $\mathfrak{F}=\langle\{w\},R\rangle$ and $\mathfrak{F}'=\langle\{w'\},R'\rangle$ with $wRw\!=\!x$ and $w'R'w'\!=\!x'$. Then $\mathfrak{F}\!\models_{\KbiG}\!\phi$ iff $\mathfrak{F}'\!\models_{\KbiG}\!\phi$ for all $\phi\in\bimodalLtriangle$.
\begin{proof}
Consider a function $h:[0,1]\rightarrow[0,1]$ s.t.\ $h(0)=0$, $h(x)=x'$, $h(1)=1$ and also satisfying $y\leq z$ iff $h(y)\leq h(z)$ and define $v'(p,w')=h(v(p,w))$. One can check by induction on $\phi\in\bimodalLtriangle$ that if $v(\phi,w)=y$, then $v'(\phi,w)=h(y)$.

If $\phi=p$, and $v(p,w)=y$, then $v'(p,w')=h(y)$ by construction. If $\phi=\chi\wedge\psi$ and $v(\chi\wedge\psi,w)=y$, let w.l.o.g.\ $v(\chi,w)=y$ and $v(\psi,w)=z\geq y$. Then, we have $v'(\chi,w')=y'$ by the induction hypothesis and $v'(\psi,w')=h(z)\geq h(y)$ by the induction hypothesis and because $y\leq z$ iff $h(y)\leq h(z)$. Hence, $v'(\chi\wedge\psi,w')=h(y)$. If $\phi=\triangle\chi$ and $v(\triangle\chi,w)=y$ observe that $y\in\{0,1\}$. Now, we have that $v(\chi,w)=1$ (if $y=1$) or $v(\chi,w)=z<1$ (if $y=0$). Hence, we have by the induction hypothesis that $v'(\chi,w')=1$ or $v'(\chi,w')=h(z)<1$, respectively. Thus, $v'(\triangle\chi,w')=1$ if $v(\triangle\chi,w)=1$ and $v'(\triangle\chi,w')=0$ if $v(\triangle\chi,w)=0$, as required. The cases of other propositional connectives can be tackled similarly.

Now let $\phi=\Box\chi$. We consider two cases: (a) $v(\Box\chi,w)=1$ and (b) $v(\Box\chi,w)=y<1$. For (a), we obtain that $v(\chi,w)=z\geq wRw=x$. Thus, applying the induction hypothesis, we have $v'(\chi,w')=h(z)\geq w'R'w'=x'$, whence, $v'(\Box\chi,w')=1$. For (b), we have $v(\chi,w)=y<wRw=x$ and by the induction hypothesis, $v'(\chi,w')=h(y)<w'R'w'=x$, whence $v'(\Box\chi,w')=h(y)$, as required.

Thus, we have that $\mathfrak{F}\not\models_{\KbiG}\phi$ implies $\mathfrak{F}'\not\models_{\KbiG}\phi$.
\end{proof}
\section{Proofs of Section~\ref{sec:KblG}}
\subsection{Theorem~\ref{theorem:embedding}, item~1\label{ssec:theorem:embedding1}}
We show that $\mathfrak{F},w\models_{\KinvG}\phi$ iff $\mathfrak{F},w\models_{\KblG}\both\rightarrow\phi^{\Join}$.
\begin{proof}
Note that it suffices to check that $\mathfrak{F},w\models_{\KinvG}\phi$ iff $\mathfrak{F},w\models^+_{\KblG}\phi^{\Join}$ since $v(\both,w)=(1,1)$. Let $v$ be a $\KinvG$ valuation on $\mathfrak{F}$ and define $v_1(p,u)=v(p,u)$ for all $p\in\Prop$ and $u\in W$ ($v_2$ can be arbitrary). We show that $v_1(\phi,u)=v(\phi^{\Join},u)$ for all $\phi\in\bimodalLinv(2)$ by induction. The basis case of variables holds by construction of $v_1$, the cases of propositional connectives hold since semantics of $\KinvG$ and $v_1$ conditions of $\KblG$ coincide (cf.~Definitions~\ref{def:KinvG} and~\ref{def:KblG}) and since $^{\Join}$ does not change propositional connectives.

Let us now consider the most instructive modal case: $\phi=\Box_2\chi$. We have $(\Box_2\chi)^{\Join}=\neg\lozenge\neg\chi^{\Join}$.
\begin{align*}
v(\Box_2\chi,u)&=\inf\{uR^-u'\rightarrow v(\chi,u')\}\tag{$\Box_2$ is associated to $R^-$}\\
&=\inf\{uR^-u'\rightarrow v_1(\chi^{\Join},u')\}\tag{IH}\\
&=\inf\{uR^-u'\rightarrow v_2(\neg\chi^{\Join},u')\}\\
&=v_2(\lozenge\neg\chi^{\Join},u)\\
&=v_1(\neg\lozenge\neg\chi^{\Join},u)
\end{align*}
Thus, we have that $\mathfrak{F},w\not\models_{\KinvG}\phi$ implies $\mathfrak{F},w\not\models^+_{\KblG}\phi^{\Join}$.

For the converse direction, we note, first of all, that all occurrences of $\neg$'s in $\phi$ are of the form $\neg\Box\neg$ and $\neg\lozenge\neg$. The proof is almost the same. Given $v_1$ on $\mathfrak{F}$, define $v^{\Join}(p,u)=v_1(p,u)$ for every $p\in\Prop$ and $u\in W$. We show that $v_1(\phi^{\Join},u)=v^{\Join}(\phi,u)$ by induction. Again, the basis case and the cases of propositional connectives are straightforward. Consider $\phi^{\Join}=\neg\Box\neg\chi^{\Join}$. Then $\phi=\lozenge_2\chi$. We proceed as follows.
\begin{align*}
v_1(\neg\Box\neg\chi^{\Join},u)&=v_2(\Box\neg\chi^{\Join},u)\\
&=\sup\{uR^-u'\wedge v_2(\neg\chi^{\Join},u')\}\\
&=\sup\{uR^-u'\wedge v_1(\chi^{\Join},u')\}\\
&=\sup\{uR^-u'\wedge v^{\Join}(\chi,u')\}\tag{IH}\\
&=v^{\Join}(\lozenge_2\chi,u)\tag{$\lozenge_2$ is associated to $R^-$}
\end{align*}
\end{proof}
\subsection{Theorem~\ref{theorem:embedding}, item~2\label{ssec:theorem:embedding}}
Let $\phi\in\bimodalLinv$, $\chi\in\fullLbilattice$, and $\Box_1$ and $\Box_2$ be associated with $R^+$ and $R^-$, respectively. Then, for every bi-relational pointed frame $\langle\mathfrak{F},w\rangle=\langle\langle W,R^+R^-\rangle,w\rangle$, it holds that $\mathfrak{F},w\models_{\KblG}\chi$ iff $\mathfrak{F},w\models_{\KinvG}\chi^\oplus\!\wedge\!{\sim}(\chi^\ominus)$.

We show by induction that the following statements hold.
\begin{enumerate}[noitemsep,topsep=2pt]
\item[$(a)$] $\mathfrak{F},w\models_{\KinvG}\chi^\oplus$ iff $\mathfrak{F},w\models^+_{\KblG}\chi$
\item[$(b)$] $\mathfrak{F},w\models_{\KinvG}{\sim}(\chi^\ominus)$ iff $\mathfrak{F},w\models^-_{\KblG}\chi$
\end{enumerate}
\begin{proof}
We define $v_1(p,u)=v(p,u)$ and $v_2(p,u)=v(p^*,u)$ for all $u\in\mathfrak{F}$. To prove ‘only if’ direction, we show that $v(\chi^\oplus,u)=v_1(\chi,u)$ and $v(\chi^\ominus,u)=v_2(\chi,u)$ by induction on $\chi$. If $\chi=p$, then $\chi^\oplus=p$ and $\chi^\ominus=p^*$, and the statement holds by construction. The cases of propositional connectives and modalities can be considered in the same manner, so we consider $\chi=-\!\psi$, $\chi=\varrho\wedge\sigma$, $\chi=\Box\psi$, and $\chi=\blacklozenge\psi$.

If $\chi=-\!\psi$, then $\chi^\oplus=\invol(\psi^\ominus)$ and $\chi^\ominus=\invol(\chi^\oplus)$. We proceed as follows.
\begin{minipage}{.49\textwidth}
\begin{align*}
v(\invol(\psi^\ominus),u)&=1-v(\psi^\ominus,u)\\
&=1-v_2(\psi,u)\tag{IH}\\
&=v_1(-\!\psi,u)
\end{align*}
\end{minipage}
\hfill
\begin{minipage}{.49\textwidth}
\begin{align*}
v(\invol(\psi^\oplus),u)&=1-v(\psi^\oplus,u)\\
&=1-v_1(\psi,u)\tag{IH}\\
&=v_2(-\!\psi,u)
\end{align*}
\end{minipage}

If $\chi=\varrho\wedge\sigma$, then $\chi^\oplus=\varrho^\oplus\wedge\sigma^\oplus$ and $\chi^\ominus=\sigma^\ominus\vee\varrho^\ominus$. We proceed as follows.

\noindent
\begin{minipage}{.49\textwidth}
\begin{align*}
v(\varrho^\oplus\!\wedge\!\sigma^\oplus,u)&=\min(v(\varrho^\oplus,u),v(\sigma^\oplus,u))\\
&=\min(v_1(\varrho,u),v_1(\sigma,u))\tag{IH}\\
&=v_1(\varrho\!\wedge\!\sigma,u)
\end{align*}
\end{minipage}
\hfill
\begin{minipage}{.49\textwidth}
\begin{align*}
v(\varrho^\ominus\!\wedge\!\sigma^\ominus,\!u)&=\max(v(\varrho^\ominus,\!u),v(\sigma^\ominus,\!u))\\
&=\max(v_2(\varrho,u),v_2(\sigma,\!u))\tag{IH}\\
&=v_2(\varrho\!\wedge\!\sigma,\!u)
\end{align*}
\end{minipage}

If $\chi=\Box\psi$, then $\chi^\oplus=\Box_1\psi^\oplus$ and $\chi^\ominus=\lozenge_2\psi^\ominus$. We proceed as follows.
\begin{align*}
v(\Box_1\psi^\oplus,u)&=\inf\limits_{u'\in W}\{uR^+u'\rightarrow v(\psi^\oplus,u')\}\tag{$\Box_1$ is associated to $R^+$}\\
&=\inf\limits_{u'\in W}\{uR^+u'\rightarrow v_1(\psi,u')\}\tag{IH}\\
&=v_1(\Box\psi,u)
\end{align*}
\begin{align*}
v(\lozenge_2\psi^\ominus,w)&=\sup\limits_{u'\in W}\{uR^-u'\wedge v(\psi^\ominus,u')\}\tag{$\lozenge_2$ is associated to $R^-$}\\
&=\inf\limits_{u'\in W}\{uR^-u'\wedge v_2(\psi,u')\}\tag{IH}\\
&=v_2(\Box\psi,u)
\end{align*}
If $\chi=\blacksquare\psi$, then $\chi^\oplus=\Box_1\psi^\oplus$ and $\chi^\ominus=\Box_2\psi^\ominus$. We proceed as follows.
\begin{align*}
v(\Box_1\psi^\oplus,u)&=\inf\limits_{u'\in W}\{uR^+u'\rightarrow v(\psi^\oplus,u')\}\tag{$\Box_1$ is associated to $R^+$}\\
&=\inf\limits_{u'\in W}\{uR^+u'\rightarrow v_1(\psi,u')\}\tag{IH}\\
&=v_1(\blacksquare\psi,u)
\end{align*}
\begin{align*}
v(\Box_2\psi^\ominus,u)&=\sup\limits_{u'\in W}\{uR^-u'\rightarrow v(\psi^\ominus,u')\}\tag{$\Box_2$ is associated to $R^-$}\\
&=\inf\limits_{u'\in W}\{uR^-u'\rightarrow v_2(\psi,u')\}\tag{IH}\\
&=v_2(\Box\psi,u)
\end{align*}
From here, we obtain for every $u\in W$ that if $v(\chi^\oplus,u)\neq1$, then $v_1(\chi,u)\neq1$ and if $v({\sim}(\chi^\ominus),u)\neq1$, then $v_2(\chi,u)\neq0$.

For the converse, given $v_1$ and $v_2$ on $\mathfrak{F}$, we define $v^{\Join}(p,u)=v_1(p,u)$ and $v^{\Join}(p^*,u)=v_2(p,u)$. Now, it suffices to show that
\begin{align*}
v^{\Join}(\chi^\oplus,u)&=v_1(\chi,u)&v^{\Join}(\chi^\ominus,u)&=v_2(\chi,u)
\end{align*}
We proceed by induction. The statement holds for $\chi=p$ by the construction of $v^{\Join}$. Since the induction steps can be demonstrated similarly, we only show $\chi=-\!\psi$, $\chi=\varrho\wedge\sigma$, and $\chi=\lozenge\psi$.

For $\chi=-\!\psi$, we have $(-\!\psi)^\oplus=\invol\psi^\ominus$ and $(-\!\psi)^\ominus=\invol\psi^\oplus$.

\begin{minipage}{.49\textwidth}
\begin{align*}
v_1(-\psi,u)&=1-v_2(\psi,u)\\
&=1-v^{\Join}(\psi^\ominus)\tag{IH}\\
&=v^{\Join}(\invol\psi^\ominus,u)
\end{align*}
\end{minipage}
\hfill
\begin{minipage}{.49\textwidth}
\begin{align*}
v_2(-\psi,u)&=1-v_1(\psi,u)\\
&=1-v^{\Join}(\psi^\oplus)\tag{IH}\\
&=v^{\Join}(\invol\psi^\oplus,u)
\end{align*}
\end{minipage}

For $\chi=\varrho\wedge\sigma$, we have $(\varrho\wedge\sigma)^\oplus=\varrho^\oplus\wedge\sigma^\oplus$ and $(\varrho\wedge\sigma)^\ominus=\varrho^\ominus\vee\varrho^\ominus$.

\noindent
\begin{minipage}{.49\linewidth}
\begin{align*}
v_1(\varrho\!\wedge\!\sigma,u)&=\min(v_1(\varrho,u),v_1(\sigma,u))\\
&=\min(v^{\Join}(\varrho^\oplus,u),v^{\Join}(\sigma^\oplus,u))\tag{IH}\\
&=v^{\Join}(\varrho^\oplus\wedge\sigma^\oplus,u)
\end{align*}
\end{minipage}
\hfill
\begin{minipage}{.49\textwidth}
\begin{align*}
v_2(\varrho\!\wedge\!\sigma,u)&=\max(v_2(\varrho,u),v_2(\sigma,u))\\
&=\max(v^{\Join}(\varrho^\ominus,u),v^{\Join}(\sigma^\ominus,u))\tag{IH}\\
&=v^{\Join}(\varrho^\ominus\vee\sigma^\ominus,u)
\end{align*}
\end{minipage}

Finally, for $\chi=\Box\psi$, we have $(\Box\psi)^\oplus=\Box_1\psi^\oplus$ and $(\Box\psi)^\ominus=\lozenge_2\psi^\ominus$.
\begin{align*}
v_1(\Box\psi,u)&=\inf\limits_{u'\in W}\{uR^+u'\rightarrow v_1(\psi,u')\}\\
&=\inf\limits_{u'\in W}\{uR^+u'\rightarrow v^{\Join}(\psi^\oplus,u')\}\\
&=v^{\Join}(\Box_1\psi^\oplus,u)\tag{$\Box_1$ is associated to $R^+$}
\end{align*}
\begin{align*}
v_2(\Box\psi,u)&=\sup\limits_{u'\in W}\{uR^-u'\wedge v_2(\psi,u')\}\\
&=\sup\limits_{u'\in W}\{uR^-u'\wedge v^{\Join}(\psi^\ominus,u')\}\\
&=v^{\Join}(\lozenge_2\psi^\ominus,u)\tag{$\lozenge_2$ is associated to $R^-$}
\end{align*}
From here, we have for every $u\in W$ that if $v_1(\chi,u)\neq1$, then $v^{\Join}(\chi^\oplus,u)\neq1$ (and thus, $v^{\Join}(\chi^\oplus\wedge{\sim}\chi^\ominus,w)\neq1$); and that if $v_2(\chi,u)\neq0$, then $v^{\Join}(\chi^\ominus,u)\neq0$, whence $v^{\Join}(\chi^\oplus\wedge{\sim}\chi^\ominus,w)=0$.
\end{proof}
\section{Proofs of Section~\ref{sec:TKinvG}}
\subsection{Theorem~\ref{theorem:TKblGcompleteness} (soundness)\label{ssec:theorem:TKblGcompleteness:soundness}}
We show that if $\mathfrak{M}$ realises the premise of a~rule, then it also realises one of its conclusions. Soundness follows from here since no closed branch is realisable.
\begin{proof}
We consider $\Rightarrow_\triangleright$ (with $\triangleright$ being $>$) and $\xcancel{\Box}_\approx$ (other rules can be tackled similarly).

Let $w:\mathbf{1}:\chi\rightarrow\psi>\mathfrak{X}$ be realised by $\mathfrak{M}$ and $\real(\mathfrak{X})=z$, i.e., $v_1(\chi\rightarrow\psi,w)>z$. Then, there are two options. First, $v_1(\psi,w)>z$. In this case, the left conclusion is realised. Second, it is possible that $v_1(\chi\rightarrow\psi,w)=1$ and $z<1$. Hence, $v_1(\chi,w)\leq v_1(\psi,w)$, whence there is some $c$ s.t.\ $v_1(\psi,w)=c$ and $v_1(\chi,w)\leq c$. Thus, the right conclusion is realised.

Assume that $w:\mathbf{1}:\xcancel{\Box}\psi\approx\mathfrak{X}$ be realised by $\mathfrak{M}$ and $\real(\mathfrak{X})=z$. Then $v_1(\xcancel{\Box}\psi,w)=z$. If $z=1$, then $\real(\mathfrak{X})=\real(1)=\real(w:\mathbf{1}:\xcancel{\Box}\psi)$ (i.e., the left conclusion is realised). Otherwise, $z\in T_1$ and $z=\max\{x\in T_1(w)\mid x\leq\inf\limits_{w'\in W}\{wR^+w'\rightarrow v_1(\psi,w')\}\}$. Thus, there is $w'$ s.t.\ $v_1(\phi,w')<w\Rmsf^+w'$ and $z\leqslant v_1(\psi,w')<\smsf(z)$ with $\smsf(z)=\min\{y\in T_1\mid y>x\}$. Hence, the right conclusion is realised.
\end{proof}
\end{document}